\documentclass[a4paper,11pt]{article}
\RequirePackage[T1]{fontenc}
\usepackage{palatino}
\usepackage{hyperref}
\usepackage[english]{babel}
\usepackage[utf8]{inputenc}
\usepackage[font={footnotesize,sf},labelfont=bf]{caption}
\usepackage{amsmath,amssymb,amsfonts,mathrsfs,amsthm}
\usepackage{graphicx,pdfpages,enumitem,cite}
\theoremstyle{plain}
\newtheorem{theorem}{Theorem}[section] 
 
\newtheorem{proposition}[theorem]{Proposition} 
\theoremstyle{plain} 
 
\theoremstyle{definition} 
 
\usepackage{geometry}
\title{Data-Assisted Non-Intrusive Model Reduction for Forced Nonlinear Finite Elements Models}
\author{M. Cenedese$^{1}$, J. Marconi$^{2}$, G. Haller$^{1}$ and S. Jain$^{3}$
\vspace{3.5mm}\\
$^{1}$Institute for Mechanical Systems, ETH Z\"urich\\ Leonhardstrasse 21, 8092 Z\"urich, Switzerland \vspace{1.5mm}\\
$^{2}$Department of Mechanical Engineering,  Politecnico di Milano \\ Via La Masa 1, 20156 Milan, Italy
\vspace{1.5mm}\\ \thanks{Corresponding author: \href{mailto:shobhit.jain@tudelft.nl}{Shobhit.Jain@tudelft.nl}}
$^{3}$ Delft Institute of Applied Mathematics, TU Delft,\\ Mekelweg 4, 2628 CD Delft, The Netherlands}
\date{\today}
\begin{document}
\maketitle
\begin{abstract}
	\noindent Spectral submanifolds (SSMs) have emerged as accurate and predictive model reduction tools for dynamical systems defined either by equations or data sets. While finite-elements (FE) models belong to the equation-based class of problems, their implementations in commercial solvers do not generally provide information on the nonlinearities required for the analytical construction of SSMs. Here, we overcome this limitation by developing a data-driven construction of SSM-reduced models from a small number of unforced FE simulations. We then use these models to predict the forced response of the FE model without performing any costly forced simulation. This approach yields accurate forced response predictions even in the presence of internal resonances or quasi-periodic forcing, as we illustrate on several FE models. Our examples range from simple structures, such as beams and shells, to more complex geometries, such as a micro-resonator model containing more than a million degrees of freedom. In the latter case, our algorithm predicts accurate forced response curves in a small fraction of the time it takes to verify just a few points on those curves by simulating the full forced-response.
\end{abstract}

\section{Introduction}
Finite element (FE) models are an invaluable tool for scientific and engineering purposes. In most industrial applications, however, FE simulations carry prohibitive computational costs. Even using dedicated commercial software, one faces major computational hurdles in predicting time-dependent response of lightly damped, nonlinear mechanical systems. To facilitate fast forced response simulations of industrial-scale mechanical structures, we will construct these nonlinear reduced-order models based on the theory of Spectral Submanifolds (SSMs) and using unforced simulation data obtained from generic FE software. 

Many model reduction techniques targeting nonlinear mechanical systems have appeared in the literature. Some methods, such as static condensation~\cite{Tiso2021}, modal truncation~\cite{GR2015,Touze2014} and  modal derivatives~\cite{Idelsohn1985a,Weeger2016,Jain2017}, fall into the category of intrusive (or direct) methods, because they require explicit access to the governing equations or the source code of an FE software to construct reduced order models (ROMs). Nonintrusive techniques~\cite{Mignolet2013}, on the other hand, have greater accessibility since they use the FE software as a black box to construct ROMs. Several data-driven techniques, including Proper Orthogonal Decomposition~\cite{Lu2019,Carlberg2011} and its deep learning-based applications~\cite{Champion2019,Gobat2023} fall into this category. In the FE context, such data-driven techniques rely on expensive full-system simulations to generate training data and have limited applicability outside the training range. Other nonintrusive techniques, such as the stiffness evaluation procedure (STEP)~\cite{Muravyov2003} and its enhancements~\cite{Perez2014,Karamooz2021}, avoid the use of expensive full trajectory simulations to obtain ROMs.

A common feature of all the above techniques is that they are projection based. Projection-based methods exhibit an inherently linear perspective on model reduction that loses mathematical justification for nonlinear systems~\cite{Haller2017}. An emerging alternative in nonlinear model reduction is the use of attracting invariant manifolds. The reduced dynamics on these low-dimensional manifolds attracts the full system’s trajectories and hence provides a mathematically rigorous ROM. Prominent examples of this approach are spectral submanifolds (SSMs)~\cite{Haller2016}, which are the smoothest nonlinear continuations of linear modal subspaces. The existence and uniqueness of SSMs are guaranteed under appropriate nonresonance conditions on the spectrum of the linearized system~\cite{Haller2016}. Indeed, these conditions can be verified via any FE package, as we will demonstrate. 

Intrusive computation of SSMs has been successfully employed to reduce various FE models of nonlinear mechanical systems~\cite{Ponsioen2018,Ponsioen2020,Jain2021,Vizzaccaro2022}, including those featuring internal resonances~\cite{Li2022a,Li2022b,Opreni2023} and parametric resonances~\cite{Thurnher2023}. More recently, the data-driven computation of SSMs has been developed~\cite{Cenedese2022a} and disseminated in the open-source packages, \texttt{SSMLearn}~\cite{SSMLearn} and \texttt{fastSSM}~\cite{Axas2022}. These developments have led to the notion of dynamics-based machine learning for nonlinearizable phenomena~\cite{Haller2022} in diverse application fields, including fluid dynamics~\cite{Kaszas2022,Haller2023} and controls~\cite{Mahlknecht2022, Alora2023}.

Motivated by the data-driven efforts of SSM computation, we aim to develop here an SSM-based nonintrusive technique for nonlinear model reduction. To this end, we will use the eigenvalues and eigenvectors of the linearized system that are provided by any generic FE solver. This linear information will be used to determine the dimension of the SSM relevant for model reduction as well as the linear part of the expansions for the SSM and its reduced dynamics. We will then perform transient simulations of the unforced mechanical structure using generic FE solvers and use the trajectory data to learn the nonlinear part of the SSM parametrization. While our SSM-based ROM will be constructed based on unforced system simulations, we will show how it yields nonlinear forced response predictions for the full system. This justifies the cost of full, unforced system simulations, which would otherwise be seen as a potentially expensive offline cost for the construction of our ROM. Furthermore, we will show that our methodology is also effective in the reduced modeling of internally resonant systems.

The remainder of this paper is organized as follows. In the next section, we define the general setup for mechanical systems and SSM-based model reduction. In Section~\ref{sec:learning}, we discuss how to learn SSMs and their reduced dynamics based on linearized system information and decaying (unforced) trajectory data. We also show how the effect of external forcing can be systematically included in SSM-based ROMs to make forced response predictions. Finally, in Section~\ref{sec:examples}, we demonstrate our methodology on FE models of one-dimensional beam structures, two-dimensional shell structures, and a three-dimensional continuum-based MEMS resonator. These examples vary in their numbers of degrees of freedom from a few hundred to more than a million. With these examples, we aim to demonstrate the data-assisted, SSM-based prediction of several nonlinearizable phenomena, which include multiple coexisting steady states and nonlinear modal interactions in internally resonant systems.

\section{Setup}
FE models for mechanics problems comprise a system of second-order ordinary differential equations for generalized displacements $\mathbf{q}(t)\in\mathbb{R}^n$ in the form
\begin{equation}
\label{eq:mechsys}
\mathbf{M}\ddot{\mathbf{q}} +  \mathbf{C}\dot{\mathbf{q}} + \mathbf{K}\mathbf{q} + \mathbf{f}^{\mathrm{int}}(\mathbf{q},\dot{\mathbf{q}}) = \varepsilon\mathbf{f}^{\mathrm{ext}}(\mathbf{q},\dot{\mathbf{q}},\boldsymbol{\Omega}t;\varepsilon),
\end{equation}
where, $\mathbf{M},\mathbf{C},\mathbf{K}\in\mathbb{R}^{n\times n}$ are the mass, damping and stiffness matrices; $\mathbf{f}^{\mathrm{int}}(\mathbf{q},\dot{\mathbf{q}})\in\mathbb{R}^n$ is the purely nonlinear internal force; and $\mathbf{f}^{\mathrm{ext}}(\mathbf{q},\dot{\mathbf{q}},\boldsymbol{\Omega}t;\varepsilon)\in\mathbb{R}^n$ is the external force with frequency vector $\boldsymbol{\Omega}\in\mathbb{R}^l$, whose amplitude is governed by the small parameter $\varepsilon > 0$. We assume that internal and external forces are smooth, and that the latter can be written in terms of its Taylor-Fourier expansion as
\begin{equation}
\label{eq:forcing_definition}
\mathbf{f}^{\mathrm{ext}}(\mathbf{q},\dot{\mathbf{q}},\boldsymbol{\Omega}t;\varepsilon)= \sum_{\mathbf{k}\in\mathbb{Z}^{l}}\mathbf{f}_{\mathbf{k}}^{\mathrm{ext}}e^{i\langle\mathbf{k},\boldsymbol{\Omega}\rangle t}+\mathcal{O}(\varepsilon\|(\mathbf{q},\dot{\mathbf{q}})\|), \qquad \mathbf{f}_{\mathbf{k}}^{\mathrm{ext}}\in\mathbb{C}^{n}, \qquad \mathbf{f}_{-\mathbf{k}}^{\mathrm{ext}}=\bar{\mathbf{f}}_{\mathbf{k}}^{\mathrm{ext}}.
\end{equation}
This forced may be autonomous (when $l=0$), periodic or quasi-periodic in $t$, depending on whether the frequencies in $\boldsymbol{\Omega}$ are rationally commensurate or not. In Eq. (\ref{eq:forcing_definition}), we denote the complex conjugate of a vector $\mathbf{z}\in\mathbb{C}^{n}$ as $\bar{\mathbf{z}}\in\mathbb{C}^{n}$.

We write system (\ref{eq:mechsys}) in a first-order form with the state vector $\mathbf{x} = (\mathbf{q},\dot{\mathbf{q}})\in\mathbb{R}^{2n}$ as
\begin{equation}
\label{eq:dynsys}
\begin{array}{c}
\dot{\mathbf{x}} = \mathbf{f}(\mathbf{x},\boldsymbol{\Omega}t;\varepsilon), \qquad \mathbf{f}(\mathbf{x},\boldsymbol{\Omega}t;\varepsilon) = \mathbf{A}\mathbf{x} + \mathbf{f}_0(\mathbf{x}) + \varepsilon\mathbf{f}_{1}(\mathbf{x},\boldsymbol{\Omega}t;\varepsilon), \\ \\ \mathbf{A}=\begin{bmatrix} \mathbf{0} & \mathbf{I} \\ -\mathbf{M}^{-1}\mathbf{K}& -\mathbf{M}^{-1}\mathbf{C} \end{bmatrix}, \,\,\,\,  \mathbf{f}_0(\mathbf{x})=\begin{pmatrix} \mathbf{0}  \\ -\mathbf{M}^{-1}\mathbf{f}^{\mathrm{int}}(\mathbf{x}) \end{pmatrix}, \,\,\,\,  \mathbf{f}_1(\mathbf{x},\boldsymbol{\Omega}t;\varepsilon)=\begin{pmatrix} \mathbf{0}  \\ \mathbf{M}^{-1}\mathbf{f}^{\mathrm{ext}}(\mathbf{x},\boldsymbol{\Omega}t;\varepsilon) \end{pmatrix},
\end{array}
\end{equation}
and we denote its trajectories starting from the initial condition $\mathbf{x}_0$ as $\mathbf{x}(t;\mathbf{x}_0,\varepsilon)$. 

We also assume that the origin is an equilibrium for system (\ref{eq:dynsys}) when $\varepsilon = 0$, and that the matrix $\mathbf{A}$ is a semi-simple matrix featuring $2n$ eigenvalues with negative real parts. As we focus on oscillatory motions, we further assume that $\mathbf{A}$ has $c\leq n$ complex conjugate pairs of eigenvalues $\lambda_1,\bar{\lambda}_1,\lambda_2,\bar{\lambda}_2,...,\lambda_{c},\bar{\lambda}_{c}$ ordered with non-increasing real parts, and we denote by $E_1,E_2,...,E_{c}$ the corresponding two-dimensional eigenspaces (or modal subspaces). We define a $2m$-dimensional (oscillatory) spectral subspace $E^{2m}$ as the direct sum of $m$ of these modal subspaces, i.e., $E^{2m} = E_{j_1} \oplus E_{j_2} \oplus ... ,\oplus E_{j_m}$, and we denote $\mathrm{spec}\left(\mathbf{A}\vert_{E^{2m}} \right) $ the set of eigenvalues related to this spectral subspace, i.e., $\mathrm{spec}\left(\mathbf{A}\vert_{E^{2m}} \right)=\left\{\lambda_{j_1},\bar{\lambda}_{j_1},\lambda_{j_2},\bar{\lambda}_{j_2},...,\lambda_{j_m},\bar{\lambda}_{j_m}\right\}$.

We recall that spectral subspaces are invariant for the linearization of system (\ref{eq:dynsys}) and that slow spectral subspaces, i.e., related to the eigenvalues with the largest real parts, are also attracting. If $2m$ slowest eigenvalues are $m$ complex conjugate pairs, then the $2m$ slow spectral subspace is $E^{2m}_S = E_{1} \oplus E_{2} \oplus ... ,\oplus E_{m}$. A generic spectral subspace $E^{2m}$ is spanned by the columns of the matrix of eigenvectors $\mathbf{V}_{E^{2m}}\in\mathbb{C}^{2n\times 2m}$, satisfying the eigenvalue problem $\mathbf{A}\mathbf{V}_{E^{2m}} = \mathbf{V}_{E^{2m}} \mathbf{R}_{E^{2m}}$ where $\mathbf{R}_{E^{2m}}\in\mathbb{C}^{2m\times 2m}$ is the diagonal matrix whose elements are those of $\mathrm{spec}\left(\mathbf{A}\vert_{E^{2m}} \right)$. We will also need the matrix $\mathbf{W}_{E^{2m}}\in\mathbb{C}^{2m\times 2n}$ the matrix satisfying the dual problem $\mathbf{W}_{E^{2m}}\mathbf{A} = \mathbf{R}_{E^{2m}}\mathbf{W}_{E^{2m}}$, and normalized such that $\mathbf{W}_{E^{2m}}\mathbf{V}_{E^{2m}} = \mathbf{I}$. We finally introduce $\mathbf{U}_0=[\mathbf{u}_{j_1}\,\,\mathbf{u}_{j_2}\,\,...\,\, \mathbf{u}_{j_m}]\in\mathbb{R}^{n\times m}$ as the matrix whose columns are the mode shapes $\mathbf{u}_j$, normalized by the mass, i.e., $\mathbf{U}_0^\top \mathbf{M} \mathbf{U}_0 = \mathbf{I}$. We have hence $\mathbf{K}\mathbf{U}_0 = \mathbf{M} \mathbf{U}_0\boldsymbol{\omega}^2_0$, where $\boldsymbol{\omega}^2_0$ is the diagonal matrix of the $m$ linear conservative natural frequencies $\omega_{0,j_1} ,\omega_{0,j_2},...\omega_{0,j_m}$.

\subsection{Spectral submanifolds and their properties}
\label{sec:ssms}
If the spectral subspace $E^{2m}$ is non-resonant (i.e., no nonnegative, low-order, integer linear combination of the spectrum of $\mathbf{A}\vert_{E^{2m}}$ is contained in the spectrum of $\mathbf{A}$ outside $E^{2m}$), then $E^{2m}$ has infinitely many nonlinear continuations in system (\ref{eq:dynsys}) for $\varepsilon$ small enough \cite{Haller2016}. These continuations are invariant manifolds of dimension $2m+l$. They are also tangent to $E^{2m}$ for $\varepsilon=0$ and have have the same quasiperiodic time dependence as $\mathbf{f}^{\mathrm{ext}}$. Out of all these invariant manifolds, we call the smoothest one the (primary) spectral submanifold (SSM) of $E^{2m}$, denoted as $\mathcal{W}_\varepsilon(E^{2m})$. For more details on the remaining, less smooth (or secondary) SSMs tangent to $E^{2m}$ for $\varepsilon=0$, see \cite{Haller2023}. In this paper, we will simply refer to the primary SSM as the "SSM" for simplifying our discussion.

An SSM-based ROM involves the descriptions of the SSM geometry in the phase space and its reduced dynamics. For the geometry, a pair of smooth, invertible maps are needed: the coordinate chart $\mathbf{y} = \mathbf{w}(\mathbf{x},\boldsymbol{\Omega}t;\varepsilon)$, which uniquely maps a state $\mathbf{x}\in\mathcal{W}_\varepsilon(E^{2m})$ into either $2m$ real reduced coordinates or $m$ complex conjugate pairs $\mathbf{y}\in\mathbb{C}^{2m}$; and the parametrization $\mathbf{x} = \mathbf{v}(\mathbf{y},\boldsymbol{\Omega}t;\varepsilon)$, which retrieves the invariant manifold in the phase space from the reduced coordinates. These two maps satisfy the invertibility relations
\begin{equation}
\label{eq:invertibility}
\mathbf{y} = \mathbf{w}(\mathbf{v}(\mathbf{y},\boldsymbol{\Omega}t;\varepsilon),\boldsymbol{\Omega}t;\varepsilon), \qquad \mathbf{x} = \mathbf{v}(\mathbf{w}(\mathbf{x},\boldsymbol{\Omega}t;\varepsilon),\boldsymbol{\Omega}t;\varepsilon).
\end{equation}
The reduced dynamics on $\mathcal{W}_\varepsilon(E^{2m})$ is described by the vector field $\dot{\mathbf{y}} = \mathbf{r}(\mathbf{y},\boldsymbol{\Omega}t;\varepsilon)$. By the invariance of the manifolds (i.e., if $\mathbf{x}_0 \in \mathcal{W}_\varepsilon(E^{2m})$ then $\mathbf{x}(t;\mathbf{x}_0,\varepsilon) \in \mathcal{W}_\varepsilon(E^{2m}) \,\, \forall t$), the mappings $\mathbf{v},\mathbf{w},\mathbf{r},\mathbf{f}$ satisfy the invariance relations
\begin{equation}
\label{eq:invariance}
\begin{array}{l}
D_{\mathbf{y}}\mathbf{v}(\mathbf{y},\boldsymbol{\Omega}t;\varepsilon) \mathbf{r}(\mathbf{y},\boldsymbol{\Omega}t;\varepsilon) + D_{\boldsymbol{\Omega}t}\mathbf{v}(\mathbf{y},\boldsymbol{\Omega}t;\varepsilon) \boldsymbol{\Omega}= \mathbf{f}(\mathbf{v}(\mathbf{y},\boldsymbol{\Omega}t;\varepsilon),\boldsymbol{\Omega}t;\varepsilon), \\
D_{\mathbf{x}}\mathbf{w}(\mathbf{x},\boldsymbol{\Omega}t;\varepsilon) \mathbf{f}(\mathbf{x},\boldsymbol{\Omega}t;\varepsilon) + D_{\boldsymbol{\Omega}t}\mathbf{w}(\mathbf{x},\boldsymbol{\Omega}t;\varepsilon) \boldsymbol{\Omega}= \mathbf{r}(\mathbf{w}(\mathbf{x},\boldsymbol{\Omega}t;\varepsilon),\boldsymbol{\Omega}t;\varepsilon).
\end{array}
\end{equation}
By the smooth dependence of the SSM on $\varepsilon$ \cite{Ponsioen2020}, we can write the expansion
\begin{equation}
\label{eq:expansion}
\begin{array}{ll}
\mathbf{w}(\mathbf{x},\boldsymbol{\Omega}t;\varepsilon)=\mathbf{W}_{0}\mathbf{x} + \mathbf{w}_{\mathrm{nl}}(\mathbf{x})+\varepsilon\mathbf{w}_{1}(\boldsymbol{\Omega}t)+\mathcal{O}(\varepsilon\|\mathbf{x}\|), & \mathbf{W}_{0}\in\mathbb{C}^{2m\times 2n}, \\
\mathbf{v}(\mathbf{y},\boldsymbol{\Omega}t;\varepsilon)=\mathbf{V}_0\mathbf{y} + \mathbf{v}_{\mathrm{nl}}(\mathbf{y})+\varepsilon\mathbf{v}_{1}(\boldsymbol{\Omega}t)+\mathcal{O}(\varepsilon\|\mathbf{y}\|),  & \mathbf{V}_{0}\in\mathbb{C}^{2n\times 2m}, \\
\mathbf{r}(\mathbf{y},\boldsymbol{\Omega}t;\varepsilon)=\mathbf{R}_0\mathbf{y} + \mathbf{r}_{\mathrm{nl}}(\mathbf{y})+\varepsilon\mathbf{r}_{1}(\boldsymbol{\Omega}t)+\mathcal{O}(\varepsilon\|\mathbf{y}\|),  & \mathbf{R}_{0}\in\mathbb{C}^{2m\times 2m},
\end{array}
\end{equation}
where $\mathbf{w}_{\mathrm{nl}}(\mathbf{x})$, $\mathbf{v}_{\mathrm{nl}}(\mathbf{y})$, $\mathbf{r}_{\mathrm{nl}}(\mathbf{y})$ are purely nonlinear maps, and $\mathbf{w}_{1}(\boldsymbol{\Omega}t)$, $\mathbf{v}_{1}(\boldsymbol{\Omega}t)$, $\mathbf{r}_{1}(\boldsymbol{\Omega}t)$ are time-dependent vectors. By the definition of SSMs, we must have 
\begin{equation}
\label{eq:props}
\mathrm{range}\left(\mathbf{V}_{0} \right) = E^{2m}, \qquad  \mathrm{spec}\left(\mathbf{R}_{0} \right) = \mathrm{spec}\left(\mathbf{A}\vert_{E^{2m}} \right),
\end{equation}
hence there exists an invertible matrix $\mathbf{P}\in\mathbb{C}^{2m\times 2m}$ such that
\begin{equation}\label{eq:propsV0R0}
	\mathbf{V}_{0} = \mathbf{V}_{E^{2m}}\mathbf{P}^{-1}
	, \qquad 
	\mathbf{R}_{0} = \mathbf{P}\mathbf{R}_{E^{2m}}\mathbf{P}^{-1}.
\end{equation}
This matrix $\mathbf{P}$ is the change of basis matrix so that $\mathbf{R}_0$ is similar to its diagonalization $\mathbf{R}_{E^{2m}}$. When we substitute the maps (\ref{eq:expansion}) into the second invariance relation of Eq. (\ref{eq:invariance}), we indeed find that $\mathbf{A} \mathbf{V}_{0} =\mathbf{V}_{0} \mathbf{R}_{0} $.
\begin{figure}[t]
    \centering
    \includegraphics[width=.8\textwidth]{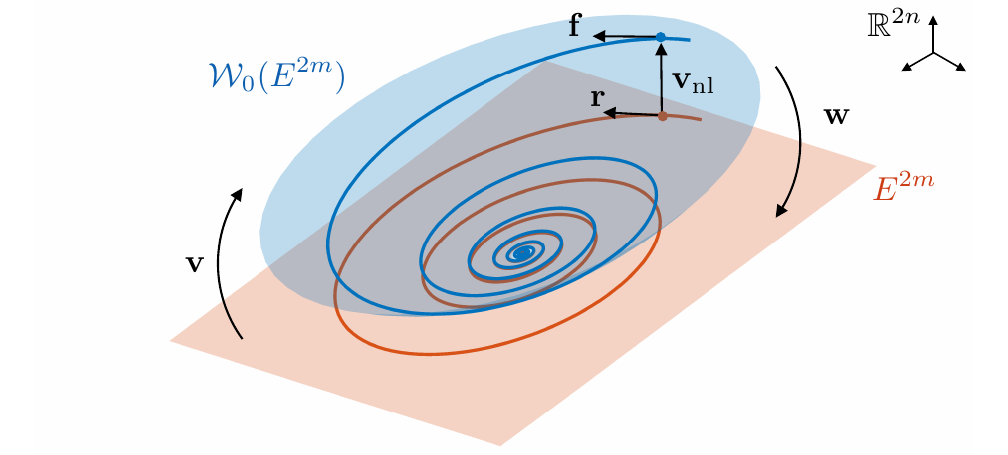}
    \caption{Illustration of parametrization of an autonomous invariant manifold $\mathcal{W}_0(E^{2m})$ (in blue, with a trajectory on it) using the tangent space at the origin, being the spectral subspace $E^{2m}$ (in orange, with the projected trajectory on it).}
    \label{fig:Fig1}
\end{figure}

\section{Learning SSMs from data}
\label{sec:learning}
We assume knowledge of the linear part of the full system (\ref{eq:mechsys}) in terms of the matrices $\mathbf{M},\mathbf{C},\mathbf{K}$, and the external forcing vectors $\mathbf{f}_{\mathbf{k}}^{\mathrm{ext}}$. This linear information is provided by any generic FE software, which we also use to generate a small number of trajectories of (\ref{eq:mechsys}) for $\varepsilon=0$. From these simulations, we aim to learn the nonlinear components of the maps in Eq. (\ref{eq:expansion}). The linear part of these maps can be explicitly obtained from the linearized system information, as we will show in this section. 

To parametrize an SSM discussed in Section \ref{sec:ssms}, we use a \emph{graph}-style of parametrization, wherein the coordinate chart is obtained as a projection onto the spectral subspace $E^{2m}$ without any time-dependent terms, as shown in Fig. \ref{fig:Fig1}. By truncating the expansions \eqref{eq:expansion} at order $\mathcal{O}(\varepsilon\|\mathbf{y}\|)$, we obtain our SSM-based ROM from the expressions
\begin{equation}
\label{eq:reducedordermodel}
\begin{array}{c}
\mathbf{w}(\mathbf{x},\boldsymbol{\Omega}t;\varepsilon)=\mathbf{W}_{0}\mathbf{x}, \\ \mathbf{v}(\mathbf{y},\boldsymbol{\Omega}t;\varepsilon)=\mathbf{V}_0\mathbf{y} + \mathbf{v}_{\mathrm{nl}}(\mathbf{y})+\varepsilon\mathbf{v}_{1}(\boldsymbol{\Omega}t), \qquad \mathbf{r}(\mathbf{y},\boldsymbol{\Omega}t;\varepsilon)=\mathbf{R}_0\mathbf{y} + \mathbf{r}_{\mathrm{nl}}(\mathbf{y})+\varepsilon\mathbf{r}_{1}(\boldsymbol{\Omega}t),
\end{array}
\end{equation}
which is valid for moderate amplitudes of displacement and and forcing~\cite{Haller2016,Breunung2018}. 

As we will show, we can obtain explicit expressions for $\mathbf{W}_{0}$, $\mathbf{V}_{0}$, $\mathbf{R}_{0}$, $\mathbf{v}_{1}(\boldsymbol{\Omega}t)$ and $\mathbf{r}_{1}(\boldsymbol{\Omega}t)$ using $\mathbf{M},\mathbf{C},\mathbf{K}$, and $\mathbf{f}_{\mathbf{k}}^{\mathrm{ext}}$. The nonlinear cores of the model $\mathbf{v}_{\mathrm{nl}}(\mathbf{y})$, $\mathbf{r}_{\mathrm{nl}}(\mathbf{y})$ can be then obtained from unforced system simulations, following \cite{Cenedese2022a,Cenedese2022b}. 

Another approach to parametrize the reduced dynamics employs an extended normal form, constructed from the data-driven approach in \cite{Cenedese2022a}. As described in detail in the upcoming Section \ref{sec:initcond}, the coefficients of this extended normal-form dynamics are sparse by construction. Furthermore, this normal-form allows the direct extraction of backbone curves, damping curves, and forced response curves (FRCs) \cite{Breunung2018,Ponsioen2019,Jain2021}. Specifically, normal forms on SSMs enable fast computation of FRCs, either via analytical solutions (for the case $m=1$) or by simplifying the periodic orbit computation to a fixed-point problem  \cite{Li2022a,Li2022b}, which is particularly useful in the case of internal resonances. 

To identify normal forms using data, we seek a near-identity change of coordinates~\cite{Cenedese2022a} 
\begin{equation}
\label{eq:invconstraints}
\begin{array}{l}
\mathbf{y}=\mathbf{h}(\mathbf{z},\boldsymbol{\Omega}t;\varepsilon)=\mathbf{P}\left( \mathbf{z} + \mathbf{h}_{\mathrm{nl}}(\mathbf{z})-\varepsilon\mathbf{h}_{1}(\boldsymbol{\Omega}t)\right), \\ \mathbf{z}=\mathbf{h}^{-1}(\mathbf{y},\boldsymbol{\Omega}t;\varepsilon)=\mathbf{P}^{-1}\mathbf{y} + \mathbf{h}^{-1}_{\mathrm{nl}}(\mathbf{P}^{-1}\mathbf{y})+\varepsilon\mathbf{h}_{1}(\boldsymbol{\Omega}t),
\end{array}
\end{equation}
that transforms the SSM-reduced dynamics in the simplest possible complex polynomial form,
\begin{equation}
\label{eq:dynnorform}
\dot{\mathbf{z}}=\mathbf{n}(\mathbf{z},\boldsymbol{\Omega}t;\varepsilon)=\mathbf{R}_{E^{2m}}\mathbf{z} + \mathbf{n}_{\mathrm{nl}}(\mathbf{z}) + \varepsilon\mathbf{n}_{1}(\boldsymbol{\Omega}t), \qquad \mathbf{z}\in\mathbb{C}^{2m},
\end{equation}
as we describe in the next sections.

Our construct remains applicable for general quasi-periodic forcing, as long as the frequency spectrum of external forcing has no resonance relationships with the spectrum of $\mathbf{A}$ outside the spectral subspace $E^{2m}$. 

\subsection{Setting up graph-style approaches}

Before data-driven learning of SS, we need to set up the linear parts of the ROM \eqref{eq:reducedordermodel} to satisfy the invertibility relations~\eqref{eq:invertibility} and the invariance equations~\eqref{eq:invariance}.

Substituting (\ref{eq:reducedordermodel}) in the first identity in Eq. (\ref{eq:invertibility}), we obtain the constrains that $\forall\,\, \mathbf{y},t$
\begin{equation}
\label{eq:inv_constr}
	\mathbf{W}_{0} \mathbf{V}_{0} = \mathbf{I}
	, \qquad 
	\mathbf{W}_{0}\mathbf{v}_{\mathrm{nl}}(\mathbf{y})\equiv \mathbf{0}
	, \qquad
	\mathbf{W}_{0}\mathbf{v}_{1}(\boldsymbol{\Omega}t) \equiv \mathbf{0}.
\end{equation}
As $\mathbf{V}_{0}$ and $\mathbf{R}_{0}$ are set according to Eq. (\ref{eq:propsV0R0}), we need to choose $\mathbf{W}_{0}$. 

The simplest choice is to set $\mathbf{W}_{0} = \mathbf{P}\mathbf{W}_{E^{2m}}$, so that the coordinate chart is a modal projection, i.e., the rows of $\mathbf{W}_{0}$ are linear combinations of those of $\mathbf{W}_{E^{2m}}$ as defined by the matrix $\mathbf{P}$. Substituting Eq. (\ref{eq:reducedordermodel}) into the first invariance equation (\ref{eq:invariance}), we obtain the reduced dynamics as  $\mathbf{r}(\mathbf{W}_{0}\mathbf{x},\boldsymbol{\Omega}t;\varepsilon)=\mathbf{W}_{0}\mathbf{f}(\mathbf{x},\boldsymbol{\Omega}t;\varepsilon)$. As expected from graph-style parameterization, this reduced dynamics is simply the projection of the full dynamics onto $E^{2m}$ via $\mathbf{W}_{0}$. In particular, we have $\mathbf{W}_{0} \mathbf{A} = \mathbf{R}_{0} \mathbf{W}_{0}$ , which upon left-multiplying with $\mathbf{V}_0$, yields
\begin{equation}
\label{eq:linearinvariance}
\mathbf{W}_{0} \mathbf{A} \mathbf{V}_{0} = \mathbf{R}_{0}.
\end{equation}
Among the possible forms of the linear parts, a simple choice is to use the first-order damped modes to parametrize the SSM. Hence, $\mathbf{P}$ is the identity matrix of dimension $2m$ in the case of complex-conjugate reduced coordinates. For real reduced coordinates, $\mathbf{P}$ has a block diagonal structure with $m$-identical blocks of matrix $\mathbf{P}_2\in\mathbb{C}^{2\times 2}$, defined as
\begin{equation}
\mathbf{P}_2 = \begin{bmatrix} 1 & -i \\ 1 & i\end{bmatrix}.
\end{equation}
For the common case of proportional damping $\mathbf{C} = \alpha \mathbf{M}+ \beta \mathbf{K}$, we simply adopt modal displacement and modal velocities from the conservative mode shapes $\mathbf{U}_0$, resulting in
\begin{equation}
\mathbf{W}_{0} = \begin{bmatrix} \mathbf{U}_0^\top \mathbf{M} & \mathbf{0} \\  \mathbf{0} & \mathbf{U}_0^\top \mathbf{M} \end{bmatrix} , \qquad \mathbf{V}_{0} = \begin{bmatrix} \mathbf{U}_0 & \mathbf{0} \\  \mathbf{0} & \mathbf{U}_0 \end{bmatrix}, \qquad  \mathbf{R}_{0} = \begin{bmatrix} \mathbf{0} & \mathbf{I} \\  -\boldsymbol{\omega}^2_0 & -(\alpha \mathbf{I} + \beta \boldsymbol{\omega}^2_0 )\end{bmatrix}.
\end{equation}
In this case, our ROM is a mechanical system in the coordinates $\mathbf{y} = (\mathbf{q}_m,\dot{\mathbf{q}}_m)$, where $\mathbf{q}_m = \mathbf{U}_0^\top \mathbf{M} \mathbf{q}$.

As an alternative, one can choose $\mathbf{W}_{0}$ more generally and not as a modal projection, i.e., in case its rows are not linear combinations of those of $\mathbf{W}_{E^{2m}}$. For example, one could use the displacements and velocities of specific degrees of freedom to describe the SSM. This approach was adopted by the pioneering work of Shaw and Pierre \cite{Shaw1993}. We will discuss this approach in one of our examples and in Appendix \ref{app:nmgs}, which also contains some cautionary notes on graph-style parametrizations whose coordinate charts are not modal projections.

\subsection{Autonomous SSM geometry and reduced dynamics}
\label{sec:initcond}
Once the spectral subspace and the consequent linear parts of the SSM and its reduced dynamics are determined, the nonlinear parts $\mathbf{v}_{\mathrm{nl}}(\mathbf{y})$ and $\mathbf{r}_{\mathrm{nl}}(\mathbf{y})$ can be computed from the simulation data via any regression technique. This is because the reduced coordinates $\mathbf{y} = \mathbf{W}_{0}\mathbf{x}$ are known a priori for our graph-style parameterization. For this purpose, we must use trajectory data that have a strong footprint of the dynamics of the SSM under investigation, up to a maximal amplitude $a_{\text{max}}$ of interest. We denote by $s(\mathbf{x})$ a function that provides the signed amplitude of interest for any state $\mathbf{x}$ of the system. 

Recognizing that an SSM is locally approximated at leading order by its spectral subspace near the fixed point, we propose two strategies for choosing initial conditions for training simulation. 
\begin{enumerate}
\item We use an external static loading to reach the amplitude of interest such that the resulting static deflection is similar to the underlying mode shape  of the SSM. While providing relevant nonlinear initial conditions, this approach requires fully nonlinear static solutions that may be computationally intensive to obtain.
\item As a faster alternative, we define initial conditions using the mode shapes as $\mathbf{x}_0=(\mathbf{q}_0,\mathbf{0})$ with  $\mathbf{q}_0=\mathbf{U}_0\mathbf{q}_m$ satisfying $|s(\mathbf{x}_0)|\geq a_{\text{max}}$. By simply evaluating the internal forces under such a displacement, we explore the nonlinear force field and choose an appropriate amplitude of the mode shape for initial conditions, as we will show with specific examples. 
\end{enumerate}

Using both initialization strategies above, we expect that the full system trajectories converge to reduced dynamics on the nearby slow SSM after some initial transients~\cite{Haller2016}. As we will show using examples, the second strategy above is also relevant for identifying intermediate SSMs, which are required for reducing internally resonant systems.

For weakly-damped mechanical systems, the reduced dynamics trajectories cover the SSM with high density. Therefore, only a few trajectories are sufficient to learn the SSM geometry. Specifically, for two-dimensional SSMs ($m=1$), a single trajectory initialized along the underlying mode shape is sufficient in our experience. For higher-dimensional SSMs, initial conditions along different modal directions are required. Specifically, for four-dimensional SSMs with $\mathbf{U}_0=[\mathbf{u}_{j_1}\,\,\mathbf{u}_{j_2}]$, we use at least three initial conditions for training, i.e., two along each of the two modes and one along their interaction $\alpha_1 \mathbf{u}_{j_1}+\alpha_2\mathbf{u}_{j_2}$, where $\alpha_1, \alpha_2 \in (0,1)$ are random numbers. Different choices for $\alpha_1,\alpha_2$ would result in more training and testing data, but the simulation time would be prohibitive, especially for very high-dimensional models.

Once the training and testing data are collected and appropriately truncated to eliminate initial transients, we use polynomial regression to identify the autonomous parts of the SSM parametrization and SSM reduced dynamics as 
\begin{equation}\label{eq:regressparared}
	\displaystyle \mathbf{v}_{\mathrm{nl}\star} = \mathrm{arg}\min_{\mathbf{v}_{\mathrm{nl}}}\sum_{j=1}^{P}\left\Vert\mathbf{x}_{j} - \mathbf{V}_0\mathbf{y}_{j}-\mathbf{v}_{\mathrm{nl}}\left(\mathbf{y}_{j}\right)\right\Vert ^{2}, \qquad
	\displaystyle \mathbf{r}_{\mathrm{nl}\star} = \mathrm{arg}\min_{\mathbf{r}_{\mathrm{nl}}}\sum_{j=1}^{P}\left\Vert\dot{\mathbf{y}}_{j} - \mathbf{R}_0\mathbf{y}_{j}-\mathbf{r}_{\mathrm{nl}}\left(\mathbf{y}_{j}\right)\right\Vert ^{2},
\end{equation}
where the reduced variables $\mathbf{y}_{j}$ are obtained directly by projecting the training data $\mathbf{x}_{j}$ onto the spectral subspace as  $\mathbf{y}_{j} = \mathbf{W}_0\mathbf{x}_{j}$; $P$ is the number of training datapoints; and the time derivative can be computed via numerical differentiation. In Appendix \ref{app:reg_para_prop}, we show that the optimal solution~\eqref{eq:regressparared} for the polynomial regression of the parametrization satisfies the second constraint in Eq.~\eqref{eq:invconstraints}. We also remark that, if the reduced dynamics has the form of a mechanical system, i.e., $\mathbf{y} = (\mathbf{q}_m,\dot{\mathbf{q}}_m)$, then the first $m$ values of the map $\mathbf{r}_{\mathrm{nl}}$ must be zero being the reduced dynamics the equivalent first order system, i.e., we only need to identify nonlinear forces.

For the dynamics in normal form, our approach assumes that the maps $\mathbf{h}_{\mathrm{nl}}$, $\mathbf{h}_{\mathrm{nl}}^{-1}$ and $\mathbf{n}_{\mathrm{nl}}$ are multivariate polynomials whose coefficients are determined by the eigenvalues of $\mathbf{R}_{E^{2m}}$ as in classic unfoldings of bifurcations \cite{GH1983,Murdock2003}. Here, the classic Poincaré \cite{Poincare1892} normal form construct is relaxed to what we refer to as \textit{extended normal form}, in which near-resonant terms are also retained in addition to the resonant terms~\cite{Ponsioen2020,Jain2021,Cenedese2022a}. The numerical values of the normal form coefficients are identified from data by minimizing the (unforced) conjugacy error as (see \cite{Cenedese2022a} for details)
\begin{equation}\label{eq:conjerror}
(\mathbf{n}_{\mathrm{nl}\star},\mathbf{h}_{\mathrm{nl}\star}^{-1})= \mathrm{arg}\min_{\mathbf{n}_{\mathrm{nl}},\mathbf{h}_{\mathrm{nl}}^{-1}}\sum_{j=1}^{P}\left\Vert D\mathbf{h}^{-1}(\mathbf{y}_{j},\mathbf{0};0)\dot{\mathbf{y}}_{j}-\mathbf{n}_0\left(\mathbf{h}^{-1}(\mathbf{y}_{j},\mathbf{0};0),\mathbf{0};0\right)\right\Vert ^{2}.
\end{equation}
Once $\mathbf{h}_{\mathrm{nl}}^{-1}$ is known, we obtain $\mathbf{h}_{\mathrm{nl}}$ via polynomial regression. Switching to polar coordinates $(\rho_k,\theta_k)$ via the transformation $z_k = \rho_k e^{i\theta_k}$ for $k = 1,2,..., m$, the general normal form on a $2m$-dimensional SSM can be inferred from~\eqref{eq:dynnorform} as
\begin{equation}\label{eq:SSMdynpolar}
\begin{array}{l}
\dot{\rho}_k = -\alpha_k(\boldsymbol{\rho},\boldsymbol{\theta})\rho_k, \\
\dot{\theta}_k = \omega_k(\boldsymbol{\rho},\boldsymbol{\theta}),
\end{array}\,\,\,\,\,\,\,\, k = 1,2,..., m, \,\,\,\,\,\,\,\, \boldsymbol{\rho} = (\rho_1,\rho_2,...\rho_m), \,\,\,\,\,\,\,\, \boldsymbol{\theta} = (\theta_1,\theta_2,...\theta_m).
\end{equation} 
Here, the zero-amplitude limits of the functions $\alpha_k$ and $\omega_k$ converge to the linearized damping and frequency of mode $j_k$. Hence, these functions represent the nonlinear continuations of linear damping and natural frequency. If the linearized frequencies are non-resonant, then $\alpha_k$ and $\omega_k$ only depend on the amplitudes $\boldsymbol{\rho}$.

\subsection{Including external forcing in the reduced-order model}
By substituting the expressions in (\ref{eq:reducedordermodel}) into the first invariance equation of (\ref{eq:invariance}) and collecting the $\mathcal{O}(\varepsilon)$-terms, we obtain
\begin{equation}
\label{eq:invarianceOe}
\mathbf{V}_{0}\mathbf{r}_{1}(\boldsymbol{\Omega}t)+D\mathbf{v}_{1}(\boldsymbol{\Omega}t)\boldsymbol{\Omega} = \mathbf{A}\mathbf{v}_{1}(\boldsymbol{\Omega}t) + \mathbf{f}_{1}(\mathbf{0},\boldsymbol{\Omega}t;0).
\end{equation}
We express $\mathbf{v}_1$ and its derivative in their Fourier expansions as
\begin{equation}
\mathbf{v}_{1}(\boldsymbol{\Omega}t) = \sum_{\mathbf{k}\in\mathbb{Z}^{l}}\mathbf{v}_{\mathbf{k}}^{1}e^{i\langle\mathbf{k},\boldsymbol{\Omega}\rangle t}, \qquad D\mathbf{v}_{1}(\boldsymbol{\Omega}t)\boldsymbol{\Omega} = \sum_{\mathbf{k}\in\mathbb{Z}^{l}}\mathbf{v}_{\mathbf{k}}^{1}i\langle\mathbf{k},\boldsymbol{\Omega}\rangle e^{i\langle\mathbf{k},\boldsymbol{\Omega}\rangle t},\qquad \mathbf{v}_{\mathbf{k}}^{1}\in\mathbb{C}^{2n},
\end{equation}
and using Eq. ~\eqref{eq:invconstraints}, we obtain $\mathbf{W}_{0}\mathbf{v}_{\mathbf{k}}^{1} = \mathbf{0}$. Hence, if we project Eq.~\eqref{eq:invarianceOe} via $\mathbf{W}_{0}$ we find that
\begin{equation}
	\label{eq:resforcing}
\mathbf{r}_{1}(\boldsymbol{\Omega}t)=\mathbf{W}_{0} \mathbf{f}_{1}(\mathbf{0},\boldsymbol{\Omega}t;0),
\end{equation}
where we used the identity $\mathbf{W}_{0} \mathbf{A} = \mathbf{R}_{0} \mathbf{W}_{0}$. Equation~\eqref{eq:resforcing} includes any resonant forcing along the mode that may trigger a nontrivial forced response. On the other hand, the $\mathcal{O}(\varepsilon)$-terms of the SSM parametrization include the effect of any nonresonant forcing. These terms are obtained by substituting Eq.~\eqref{eq:resforcing} into Eq.~\eqref{eq:invarianceOe} leading to
\begin{equation}\label{eq:paramforcingnonmodal}
\mathbf{v}_{\mathbf{k}}^{1} = (\mathbf{A} - i\langle\mathbf{k},\boldsymbol{\Omega}\rangle \mathbf{I})^{-1}( \mathbf{V}_{0}\mathbf{W}_{0} - \mathbf{I} )\begin{pmatrix}\mathbf{0} \\ \mathbf{M}^{-1}\mathbf{f}_{\mathbf{k}}^{\mathrm{ext}} \end{pmatrix}.
\end{equation}
Using modal coordinates and denoting $\mathbf{w}^*$ as the complex-conjugate transpose of the vector $\mathbf{w}$, we obtain an alternative expression for Eq.~\eqref{eq:paramforcingmodal} as
\begin{equation}
\label{eq:paramforcingmodal}
\mathbf{v}_{\mathbf{k}}^{1} = \sum_{\substack{ j = 1, \\ \lambda_j \notin \mathrm{spec}\left(\mathbf{A}\vert_{E^{2m}}\right)} }^{2n}\frac{1}{\lambda_j - i\langle\mathbf{k},\boldsymbol{\Omega}\rangle }\mathbf{v}_j \mathbf{w}_j^*( \mathbf{V}_{0}\mathbf{W}_{0} - \mathbf{I} )\begin{pmatrix}\mathbf{0} \\ \mathbf{M}^{-1}\mathbf{f}_{\mathbf{k}}^{\mathrm{ext}} \end{pmatrix},
\end{equation}
which is useful for computational implementations. In particular, for a mechanical system with proportional damping, we use the conservative, mass-normalized mode shapes to rewrite Eq.~\eqref{eq:paramforcingmodal} with $\mathbf{v}_{\mathbf{k}}^{1} = (\mathbf{v}_{\mathbf{k},q}^{1},\mathbf{v}_{\mathbf{k},\dot{q}}^{1})$ as
\begin{equation}
	\label{eq:paramforcingmodalMech}
	\mathbf{v}_{\mathbf{k},q}^{1} = \sum_{\substack{ j = 1, \\ \omega_{0,j} \notin \boldsymbol{\omega}_0 } }^n\frac{\mathbf{u}_j^\top \mathbf{f}_{\mathbf{k}}^{\mathrm{ext}}}{\omega_{0,j}^2 - \langle\mathbf{k},\boldsymbol{\Omega}\rangle^2 + i(\alpha + \beta \omega_{0,j}^2 )\langle\mathbf{k},\boldsymbol{\Omega}\rangle }\mathbf{u}_j , \qquad \mathbf{v}_{\mathbf{k},\dot{q}}^{1} = i\langle\mathbf{k},\boldsymbol{\Omega}\rangle\mathbf{v}_{\mathbf{k},q}^{1}.
	\end{equation}
Indeed, Eq. (\ref{eq:paramforcingmodalMech}) may be approximated by using a finite number of non-resonant modes $N\ll n$. 

When the reduced dynamics are transferred to their normal form, further steps are necessary to include the effect of external forcing. First, we remark that the quasi-periodic vector $\mathbf{h}_{1}(\boldsymbol{\Omega}t)$ appears on both $\mathbf{h}$, $\mathbf{h}^{-1}$ with opposite sign due to invertibility relation $\mathbf{y}=\mathbf{h}(\mathbf{h}^{-1}(\mathbf{y},\boldsymbol{\Omega}t;\varepsilon),\boldsymbol{\Omega}t;\varepsilon)$. Here, we focus on the periodic forcing case, which appears commonly in structural dynamics applications (see \cite{Jain2021,Cenedese2022a} for details on the treatment for generic quasi-periodic forcing). Thus, we now assume that the external forcing is given as
\begin{equation}\label{eq:externalperiodicforcing}
\mathbf{f}^{\mathrm{ext}}(\mathbf{q},\dot{\mathbf{q}},\boldsymbol{\Omega}t;\varepsilon)= \mathbf{f}_{0}^{\mathrm{ext}}\frac{e^{i\Omega t}+e^{-i\Omega t}}{2}+\mathcal{O}(\varepsilon\|(\mathbf{q},\dot{\mathbf{q}})\|).
\end{equation}
We assume, without loss of generality, that the diagonal entries of $\mathbf{R}_{E^{2m}}$ are ordered as \linebreak $\{ \lambda_{j_1},\,\lambda_{j_2},\,...\,\lambda_{j_m},\,\bar{\lambda}_{j_1},\,\bar{\lambda}_{j_2},\,...\,\bar{\lambda}_{j_m}\}$. Hence, the last $m$ columns (resp. rows) of $\mathbf{P}$ (resp. $\mathbf{P}^{-1}$) are the complex conjugates of the first $m$. We also denote $\boldsymbol{\Lambda}_m$ as the diagonal matrix whose entries are $\{ \lambda_{j_1},\,\lambda_{j_2},\,...\,\lambda_{j_m}\}$. Finally we introduce the following notation:
\begin{equation}
	\mathbf{g}^r = \mathbf{P}^{-1}\mathbf{W}_{0} \begin{pmatrix} \mathbf{0} \\\mathbf{f}_{0}^{\mathrm{ext}} \end{pmatrix}= \begin{pmatrix} \mathbf{g} \\ \bar{\mathbf{g}} \end{pmatrix}, \qquad \mathbf{g} = \left( g_1, \, g_2, \,  ... \,, \, g_m \right)^\top, \qquad g_k\in\mathbb{C} \,\,\mathrm{for}\,\, k=1,\,2,\,...\,,\,m.
\end{equation}
For our reduced dynamics to capture any possible resonant forcing, we focus on the case wherein $\Omega$ is close to the natural frequency of $K$ of the modes associated with the SSM. Specifically, we define the index set $R :=\{k_1,k_2,...,k_K\}$ with $1\leq K\leq m$ 
satisfying $k\in R:\omega_{0,k}\approx \Omega$. If the linearized frequencies of the modes related to the SSM are well-separated, then $R$ contains only one element, but $R$ may contain multiple entries otherwise. For instance, the set $R$ contains two indices when there exists a $1:1$ internal resonance. Let $\mathbf{I}_R\in\mathbb{R}^{m \times m}$ be a diagonal matrix such that
\begin{equation}
	\begin{cases}
	\left(\mathbf{I}_R\right)_{kk} = 1,\qquad \mathrm{if\,\,} k\in R \\
	\left(\mathbf{I}_R\right)_{kk} = 0,\qquad \mathrm{otherwise}.	
	\end{cases}
\end{equation}
As we show in Appendix \ref{app:s1}, the forcing terms in the normal form are given as
\begin{equation}
	\mathbf{n}_1(\Omega t) = \begin{pmatrix} e^{i\Omega \mathbf{I} t} \mathbf{I}_R \mathbf{g} \\ e^{-i\Omega \mathbf{I} t} \mathbf{I}_R \mathbf{g} \end{pmatrix} , \,\,\,\,\,\, \mathbf{h}_1(\Omega t) = \begin{pmatrix} \left(\boldsymbol{\Lambda}_m - i\Omega \mathbf{I} \right)^{-1}e^{i\Omega \mathbf{I} t} \left(\mathbf{I}-\mathbf{I}_R \right)\mathbf{g} + \left(\boldsymbol{\Lambda}_m + i\Omega \mathbf{I} \right)^{-1}e^{-i\Omega \mathbf{I} t}\mathbf{g} \\  \left(\bar{\boldsymbol{\Lambda}}_m + i\Omega \mathbf{I} \right)^{-1}e^{-i\Omega \mathbf{I} t} \left(\mathbf{I}-\mathbf{I}_R \right)\bar{\mathbf{g}} + \left(\bar{\boldsymbol{\Lambda}}_m - i\Omega \mathbf{I} \right)^{-1}e^{i\Omega \mathbf{I} t}\bar{\mathbf{g}} \end{pmatrix}.
\end{equation}
Expressing $g_k = if_ke^{i\phi_k}$, where $f_k = |g_k|$ denotes the modal forcing amplitude of the $j_k$-th mode, and $ \phi_k = \angle g_k - \pi/2$ denotes its phase (with $\angle$ being the argument of the complex number $g_k$), we obtain the polar normal form including forcing terms as
\begin{equation}\label{eq:SSMdynpolarforced}
	\begin{array}{l}
	\dot{\rho}_k = -\alpha_k(\boldsymbol{\rho},\boldsymbol{\theta})\rho_k - f_k \sin\left( \Omega t + \phi_k - \theta_k \right), \\
	\displaystyle \dot{\theta}_k = \omega_k(\boldsymbol{\rho},\boldsymbol{\theta}) + \frac{f_k}{\rho_k} \cos\left( \Omega t + \phi_k - \theta_k \right),
	\end{array}
\end{equation} 
for every $k\in R$. For $k\notin R$, the normal form does not contain any forcing terms as in Eq.~\eqref{eq:SSMdynpolar}. As mentioned earlier, the polar normal forms are instrumental in expressing the periodic orbit computation as a fixed-point problem via appropriate phase shifts. For $ m = 1$, the forced response can also be retrieved analytically as the zero-level set of a scalar function~\cite{Breunung2018,Jain2021,Cenedese2022a}. For internally resonant systems ($m>1$), forced response can still be obtained as a fixed-point problem~\cite{Li2022a,Li2022b}, which greatly simplifies bifurcation analysis, as we demonstrate through examples in the next section.

\section{Examples}\label{sec:examples}
We now illustrate our data-assisted SSM-reduction approach on specific finite element models. Our computations have been carried out using the open-source \textsc{MATLAB}\textsuperscript{©}
packages, \texttt{SSMLearn} \cite{Cenedese2022a} and \texttt{SSMTool} \cite{SSMTool2021}. We use \texttt{SSMLearn} to identify ROMs from data, whose results are compared to the equation-driven ROMs obtained using \texttt{SSMTool}. Moreover, we have incorporated some features of \texttt{SSMTool} in \texttt{SSMLearn} for the prediction of the forced response, especially in internally resonant cases \cite{Li2022a}. Our examples and results are available in the \texttt{SSMLearn} repository at \cite{SSMLearn}.

To present our results for forced periodic motions of period $T$, we use the following definitions for the amplitude and the phase of oscillations:
\begin{equation}\label{eq:physampphase}
\mathrm{amp} = \max_{t \in [0,T)} \left| s \left( \mathbf{v} \left( \mathbf{y}(t), \boldsymbol{\Omega}t \right)\right)\right|, \qquad \mathrm{phase} = \angle \int_0^T s \left( \mathbf{v} \left( \mathbf{y}(t), \boldsymbol{\Omega}t \right)\right) e^{-i2\pi t/T} dt.
\end{equation} 
Unless specified otherwise, we will consider the case $T = 2\pi / \Omega$, and, therefore, the phase in Eq. (\ref{eq:physampphase}) is that of the primary harmonic. For the case of backbone curves associated to two-dimensional SSMs, we use the amplitude metric~\cite{Szalai2017,Ponsioen2018,Cenedese2022b}
\begin{equation}\label{eq:NFphysamp}
	\mathrm{amp}(\rho) = \max_{\theta \in [0,2\pi)} \left| s \left( \mathbf{v} \left( \mathbf{h}(\mathbf{z},\mathbf{0};0),\mathbf{0};0 \right)\right)\right|, \qquad \mathbf{z} = \left( \rho e^{i\theta},\rho e^{-i\theta} \right).
\end{equation} 
%

Following~\cite{Cenedese2022a,Cenedese2022b}, we use the normalized mean-trajectory-error ($\mathrm{NMTE}$) to quantify the errors of an SSM-based ROM in autonomous trajectory predictions. For $P$ observations  $\mathbf{x}_{j}, j=1,\dots, P$ along a trajectory, and their model-based reconstructions,
$\hat{\mathbf{x}}_{j}$, this modeling error in percentage is defined as 
\begin{equation}
\mathrm{NMTE}=\frac{100}{P\|\underline{\mathbf{x}}\|}{\displaystyle \sum_{j=1}^{P}\left\Vert \mathbf{x}_{j}-\hat{\mathbf{x}}_{j}\right\Vert }\,,\label{eq rRMSE}
\end{equation}
where $\underline{\mathbf{x}}$ is a relevant normalization vector. For example, $\underline{\mathbf{x}}$ may be the data point with the largest norm. We note that using higher-order polynomials generally reduces the $\mathrm{\mathrm{NMTE}}$ error to any required level but excessively small errors can lead to overfitting. In our examples, we will consider acceptable model when featuring $\mathrm{\mathrm{NMTE}}$ errors on test data in the order of $1\%-10\%$, favoring lower order models to higher ones. To validate our SSM-based predictions of frequency response curves (FRCs), we use \texttt{SSMTool} or direct numerical integration.
\subsection{Von Kármán beam}
\label{sec:examplebeam}
\begin{figure}[h!]
    \centering
    \includegraphics[width=1\textwidth]{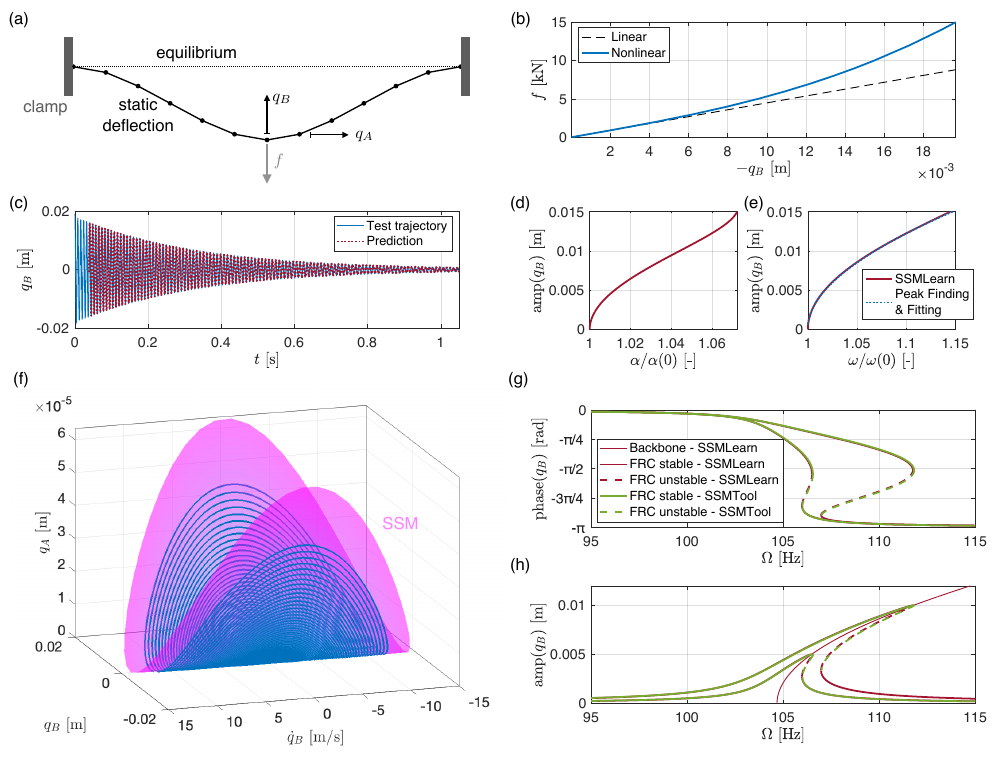}
    \caption{Plot (a) shows the finite element discretization of the beam, including the equilibrium position and the static deflection when subject to midpoint loading. The force displacement relation of such static loading is instead shown in plot (b), distinguishing the linear (black dashed line) and nonlinear case (blue line), and plotting the midpoint displacement. Plot (c) shows the test trajectory from the numerical simulation of the full model (blue curve) and its prediction (red line) from the \texttt{SSMLearn} reduced-order model. This reduced-order model predicts the backbone curves shown in plot (d,e), in terms of damping and frequency, where the frequency is compared to that extracted by processing the training trajectory with the method of Peak Finding and Fitting, \cite{Jin2020}. Plot (f) shows the SSM in the physical space along with the training trajectory, where $q$ and $q_a$ are the transverse and longitudinal displacements shown in plot (a), respectively. Plots (g,h) show forced responses in terms of amplitude and phase of $q$ computed via \texttt{SSMTool} (green) and \texttt{SSMLearn} (dark red) for two forcing amplitude values.}
    \label{fig:vonkarmanbeam}
\end{figure}

As a first example, we analyze an FE model of a von Kármán beam~\cite{Jain2018} with clamped-clamped boundary conditions, shown in Fig. \ref{fig:vonkarmanbeam}(a), which is also discussed in \cite{Cenedese2022a}. This beam model captures moderate deformations by including a nonlinear, quadratic term in the kinematics. Here, we discretize the beam with 12 elements, using cubic shape functions for the transverse deflection and linear shape functions for the axial displacement. The resulting model contains 33 degrees of freedom (DOFs), and it describes an aluminium beam of length 1 [m], width 5 [cm], thickness 2 [cm] and the material damping modulus 10$^6$ [Pa-sec]. The slowest eigenvalue is approximately $-3.09+i657.72$. 

The spectral gap (ratio between the real parts) between the first and second slowest modes is 7.6, which means that the decay along the second or higher modes is more than seven times faster than that along the first mode. Hence, we aim to construct a ROMs for this beam using the slowest, two-dimensional SSM, which is the nonlinear continuation of the first vibration mode. 

To capture data close to the SSM, we used the initialization strategy based on static loading outlined in Section \ref{sec:initcond}. Specifically, we force the beam from midpoint loading $f$ and we measure the transverse displacement of this midpoint $q_B$, as shown in Fig.~\ref{fig:vonkarmanbeam}a. In Fig.~\ref{fig:vonkarmanbeam}b, we compare the linear and nonlinear internal force at the midpoint vs the static deflection as the force magnitude increases. This choice of forcing results in a static displacement similar to the first mode shape. Hence, we expect that an unforced trajectory initialized with such a displacement would quickly converge to the SSM. We choose the initial amplitudes for training and testing trajectories near a displacement $q_B$ of around 2 [mm]. Then, we train a seventh-order model using \texttt{SSMLearn}, whose reduced dynamics are obtained as
\begin{equation}\label{eq:nf_vkbeam}
	\begin{array}{rl}\dot{\rho} = & \alpha(\rho) = -3.09\rho -1.6198 \rho^{3}+2.696 \rho^{5}+0.83303 \rho^{7}\\ \dot{\theta} = &\omega(\rho) = +657.7165+469.4784 \rho^{2}-308.8319 \rho^{4}-103.9608 \rho^{6}.\end{array}
\end{equation}

This ROM returns an NMTE of 3.53 \% for the test trajectory, with the prediction shown in Fig. \ref{fig:vonkarmanbeam}(c). The ROM also produces the backbone curves for damping and frequency in Figs.~\ref{fig:vonkarmanbeam}(d,e). The $x$-axis of these backbone curves denotes the variation of damping and frequency with respect to the zero-amplitude limit, i.e., their linearized values. The instantaneous frequency backbone of  Fig. \ref{fig:vonkarmanbeam}(e) has a good agreement with those extracted directly from the training trajectory using the signal processing method known as Peak Finding and Fitting~\cite{Jin2020}. The ROM also describes the SSM geometry, as shown in Fig.~\ref{fig:vonkarmanbeam}(f), along with the training trajectory. The plot in Fig.~\ref{fig:vonkarmanbeam}(f) is in the coordinates $q_B$, $\dot{q}_B$ and $q_A$, where the axial displacement $q_A$ of the midpoint is plotted against its bending displacement and velocity $q_B$, $\dot{q}_B$. Hence, the geometry of the SSM captures the nonlinear bending-stretching coupling of the beam.

For modeling the SSM, we can also use the non-modal graph-style parametrization discussed in Appendix \ref{app:nmgs} and choose $\mathbf{W}_0$ to be the projection to the midpoint displacement $q_B$ and its velocity $\dot{q}_B$. In this case, we obtain a similar NMTE using a cubic order model, whose reduced dynamics take the form
\begin{equation}\label{eq:mech_vkbeam}
    \begin{array}{rl}
    \ddot{q}_B = &  -432600 q_B -6.18 \dot{q}_B +2548.1  q_B^2 + 15.476  q_B\dot{q}_B  -0.012043 \dot{q}_B^2 \\ & -743340000 q_B^3      -21020  q_B^2 \dot{q}_B   -339.02  q_B \dot{q}_B^2 + 0.0047173\dot{q}_B^3.
    \end{array}
\end{equation}
Note that, however, the characterization of this reduced dynamics is not as immediate as that of the normal form from a dynamical system perspective. Moreover, as discussed in Appendix \ref{app:nmgs}, forcing is not as simple to add to the reduced as for the modal graph-style.

Now, we study the forced response of the system using our SSM-based ROM in normal form. We introduce periodic forcing to this ROM by simply turning the midpoint static forcing into time-periodic. As shown in Fig. \ref{fig:vonkarmanbeam}(g,h), the FRC predictions of the \texttt{SSMLearn} model are in close agreement with the equation-driven model obtained via \texttt{SSMTool}.
\begin{figure}[h!]
    \centering
    \includegraphics[width=1\textwidth]{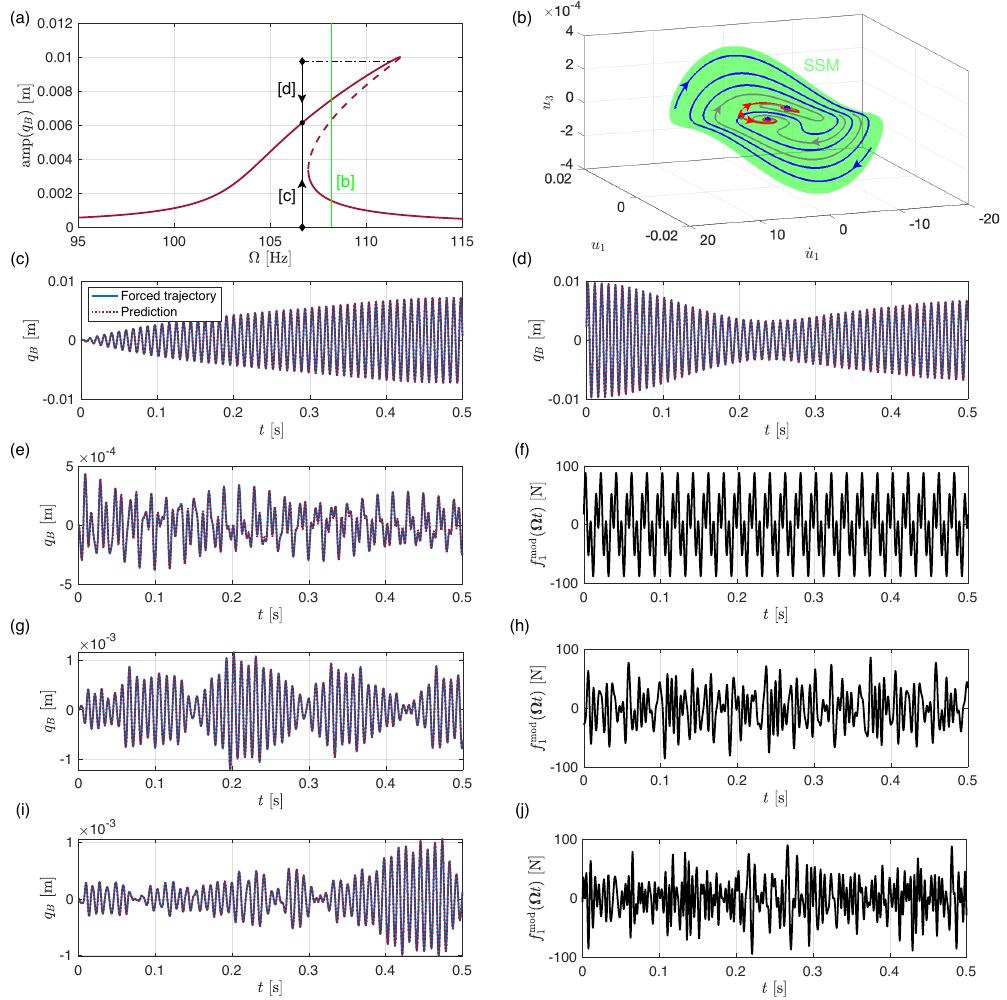}
    \caption{Plot (a) shows the forced response curve, whose green line depict the frequency to which plot (b) refers, where a section of the three dimensional forced SSM is shown, along with the periodic orbits (red and blue dots), the stable (blue line) and unstable (red line) manifolds of the saddle (red point), and two trajectories (grey lines) coverging to the two attractors. Plots (c,d) show two periodically forced trajectories and their predictions converging to the same attractor from different initial conditions, as shown in plot (a). Plots (e,g,i) show three quasi-periodically forced trajectories with initial condition being the origin, whose forcing is shown in plots (f,h,j), respectively.}
    \label{fig:vonkarmanbeamforced}
\end{figure}

Next, we turn our attention to additional forced response. Namely, instead of looking at periodic solutions arising from periodic forcing, we show that the reduced-order model we constructed is capable of predicting general forced trajectories, as show in Fig. \ref{fig:vonkarmanbeamforced}. First, by still using periodic forcing, we simulate two trajectories that converge to the periodic attractor of the forced response curve at $\Omega =105$ [Hz]. The first trajectory, shown in Fig. \ref{fig:vonkarmanbeamforced}(c), has initial condition at the unforced beam equilibrium and converges to the attractor as shown in Fig. \ref{fig:vonkarmanbeamforced}(a). The second trajectory, shown in Fig. \ref{fig:vonkarmanbeamforced}(d), still converges to the same attractor but from a higher amplitude initial condition obtained as a point of the high amplitude attractor at  $\Omega =110$ [Hz], as indicated in Fig. \ref{fig:vonkarmanbeamforced}(a). From these numerical experiments, we note that the \texttt{SSMLearn} predictions agree with the full system simulations. Instead, Fig. \ref{fig:vonkarmanbeamforced}(b) shows a section within the modal coordinate space of the three-dimensional SSM at $\Omega = 108$ [Hz], which deviates from the two-dimensional autonomous one when periodic forcing is added to the system. In this manifold, the three periodic solutions (appearing as points) are displayed, along with the stable and unstable manifolds of the saddle-type periodic orbit of the frequency response. 

The remaining plots in Fig.~\ref{fig:vonkarmanbeamforced} focus on quasi-periodic forcing, where simulations with forcing in (f,h,j) are compared to predictions in (e,g,i). All the trajectories have the origin as the initial condition, and the forcing phases are generated randomly. The trajectory in \ref{fig:vonkarmanbeamforced}(e) features forcing with two frequencies, i.e., $\boldsymbol{\Omega}=(50,150)$ [Hz]; the in Fig.~\ref{fig:vonkarmanbeamforced}(g) has 20 forcing frequencies equally spaced between 50 Hz and 200 Hz; and that in Fig.~\ref{fig:vonkarmanbeamforced}(i) is subject to 200 forcing frequencies equally spaced between 50 Hz and 250 Hz. Overall, the predictions of the ROM match closely with the full system simulations.

\subsection{Prismatic beam in 1:3 internal resonance}
\label{sec:exampleprismatic}

As our next example, we consider the forced hinged-clamped beam discussed in \cite{Li2022a}, originally presented by \cite{Nayfeh1974}. After non-dimensionalization, modal expansion and Galerkin projection, the governing PDE becomes a system of ODEs of the form
\begin{equation}
	\ddot{u}_j + \omega_j^2 u_j = \delta\left(-2c \dot{u}_j +\frac{1}{2l} \sum_{k,l,m =1}^n \alpha_{j,k,l,m} u_k u_l u_m\right) +\varepsilon f_j \cos(\Omega t)
\end{equation}
for $j = 1, \, 2, \, ... ,\, n$, where $u_j$ and $\omega_j$ are the modal coordinates and their respective eigenfrequencies; $\delta$ is the dimensionless slenderness ratio; $c$ is the dimensionless damping coefficient; $l$ is the ratio between the beam length and its characteristic length; $\alpha_{j,k,l,m}$ are the coefficients of cubic nonlinearities; and $f_j$ are the forcing coefficients. For additional details, we refer to reader to \cite{Nayfeh1974,Li2022a}. For $l=2$, the first two modes exhibit a frequency ratio $\omega_2 \approx 3\omega_1$, where $\omega_1 \approx 3.8553$ and $\omega_2 \approx 12.4927$. Hence, these modes are nearly in $1:3$ internal resonance. To study this system, we perform a reduction to its slow, four-dimensional SSM based on the internally resonant modes. Following~\cite{Li2022a}, we consider $n=10$, $\varepsilon = \delta = 10^{-4}$, $c=100$, $ \varepsilon f_1 = 2, \,3.5$ and $f_2 = f_3 = ... = f_{10} = 0$.
\begin{figure}[h!]
    \centering
    \includegraphics[width=1\textwidth]{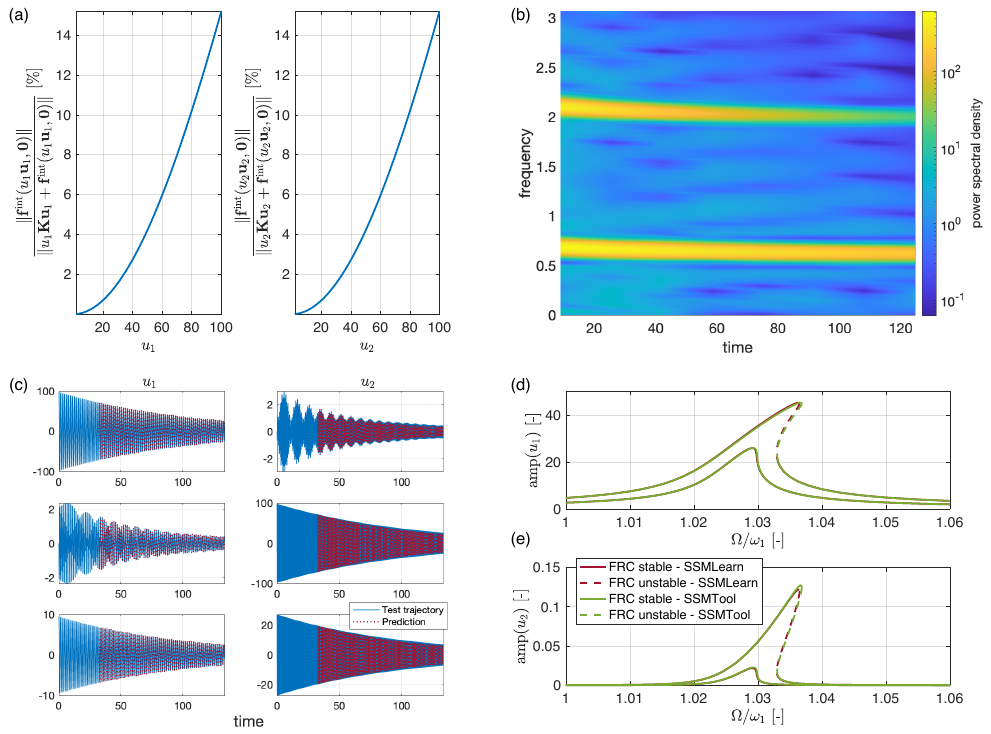}
    \caption{Plot (a) shows the evaluation of the internal force field nonlinearities under varying imposed modal displacement field, for the first (left) and second mode (right). Plot (c) shows the decaying trajectories in the two modal coordinates obtained from the full-order model (in blue) and their prediction (dark red) from the data-driven reduced-order model and plot (b) shows the spectrogram trajectory in blue on the top left plot of (c). Plots (d,e) show the amplitudes of the first two modal coordinates of forced responses computed via \texttt{SSMTool} and \texttt{SSMLearn} for two forcing amplitude values near the lowest eigenfrequency.}
    \label{fig:prismaticbeam}
\end{figure}

To obtain the initial conditions for generating the training trajectories, we use the second strategy outlined in Section \ref{sec:initcond}, i.e., we impose modal displacement fields for the first two modes. In Fig. \ref{fig:prismaticbeam}(a), we evaluate the nonlinear static force fields and plot the ratio between the norm of the nonlinear static force field and that of the full static force field (linear and nonlinear), vs. the amplitude of the imposed modal displacement field, for the first and second modal amplitudes. The plots show a nonlinear trend over the depicted range of modal displacements. Within this range, we use initial conditions for training trajectories from the set $\mathcal{D}_{\mathrm{IC, \,training}}:=\{u_{1,0}\mathbf{u}_1,\, u_{2,0}\mathbf{u}_2, \, a u_{1,0}\mathbf{u}_1 + b u_{2,0}\mathbf{u}_2 \}$, where we choose $u_{1,0} = u_{2,0} = 100$, while $a,\,b$ are random numbers in the interval $(0,\,1)$. 
Fig. \ref{fig:prismaticbeam}(b) shows the spectrogram of the unforced decay of the first modal coordinate for the trajectory with initial condition $\mathbf{x}(0) = [u_{1,0}\mathbf{u}_1,\,\mathbf{0}]$, where we observe that the first and the third harmonics dominate the response.

To generate data for testing, we simulate trajectories with initial conditions in the set $\mathcal{D}_{\mathrm{IC, \,test}}=0.95 \mathcal{D}_{\mathrm{IC, \,training}}$. Cubic-order models both for the SSM parametrization and for its reduced dynamics are sufficiently accurate to predict test trajectories, as shown in Fig. \ref{fig:prismaticbeam}c . Indeed, we obtain an optimal NMTE of $9.74$ \% on the test set, which does not reduce upon increasing the polynomial order further than 3. The reduced dynamics in the normal form parametrization takes the form
\begin{equation}\label{eq:nf_prismaticbeam}
	\begin{array}{rl}\dot{\rho}_{1} = &-0.01\rho_{1}-5.0739e-05 \rho^{3}_{1}+0.00049014 \rho_{1}\rho^{2}_{2}\\ &+\mathrm{Re}((0.031335-0.58438i) \rho^2_{1} \rho_{2} e^{i(-3 \theta_{1}+ \theta_{2})}),\\ 
\dot{\rho}_{2} = &-0.01\rho_{2}-0.0036841 \rho^{2}_{1}\rho_{2}+1.5171e-05 \rho^{3}_{2}\\ &+\mathrm{Re}((-0.0040419-0.51863i) \rho^{3}_{1} e^{i(+3 \theta_{1}- \theta_{2})}),\\ 
\dot{\theta}_{1} = &+3.8553+4.3244 \rho^{2}_{1}+1.1339 \rho^{2}_{2}\\ &+\mathrm{Im}((0.031335-0.58438i) \rho_{1} \rho_{2} e^{i(-3 \theta_{1}+ \theta_{2})}),\\ \dot{\theta}_{2} = &+12.4927+3.6813 \rho^{2}_{1}+1.8562 \rho^{2}_{2}\\ &+\mathrm{Im}((-0.0040419-0.51863i) \rho^{3}_{1} \rho^{-1}_{2} e^{i(+3 \theta_{1}- \theta_{2})}).
	\end{array}
\end{equation}
%
We also remark that \texttt{SSMTool} obtains a normal form of the type in \eqref{eq:nf_prismaticbeam}, but with potentially different coefficients from those of \texttt{SSMLearn}, since the normalizations and transformations present in the implementation influence their values.

In Fig. \ref{fig:prismaticbeam}(d,e), we compare the FRCs predicted using the forced version of our ROM~\eqref{eq:nf_prismaticbeam} with those obtained via \texttt{SSMTool} that have already been validated against the full system in \cite{Li2022a}. In summary, our SSM-based ROM, trained using unforced data, makes accurate predictions of the forced response and its bifurcations with respect to the forcing frequency in this internally resonant system.

\subsection{Von Kármán shells with and without 1:2 internal resonance}
\label{sec:exampleshells}
We now consider the shallow-arc example discussed in \cite{Jain2018b,Jain2021,Li2022a}. The model is a FE discretization of a geometrically nonlinear shallow shell structure, shown in Fig. \ref{fig:vonkarmanshellsconfig}(a), which is simply supported at the two opposite edges aligned along the $y$-axis of Fig. \ref{fig:vonkarmanshellsconfig}(a).
\begin{figure}[h!]
    \centering
    \includegraphics[width=1\textwidth]{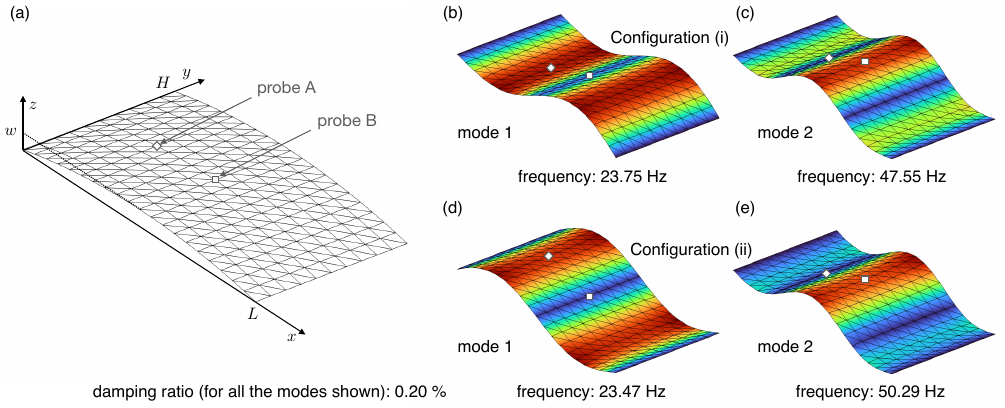}
    \caption{Plot (a) shows the finite element discretization of the shallow-arc reference position and the two probe points whose transverse displacements are denoted $q_A$ and $q_B$. Plots (b,c) show the displacement field for the first two conservative mode shapes of configuration (i) without internal resonance, while plots (d,e) show these shapes for configuration (ii) with $1:2$ internal resonance. The plots also show the natural frequencies in Hz and the damping ratios.}
    \label{fig:vonkarmanshellsconfig}
\end{figure}

 Following prior works~\cite{Jain2021,Li2022a}, we choose the material density as $2700$ [kg/m$^3$], Young's modulus as $70$ [GPa], Poisson's ratio as $0.33$, the length $L = 2$ [m], the width $H = 1$ [m] and the thickness $0.01$ [m]. We consider two different values for the curvature parameter $w$ (see Fig. \ref{fig:vonkarmanshellsconfig} (a)), that is, (i) $w = 0.1$ [m]~\cite{Jain2021} and (ii) $w = 0.041$ [m]. These two values lead to different resonance configurations: (i) without internal resonance~\cite{Jain2021} and (ii) with $1:2$ internal resonance between the first two modes~\cite{Li2022a}. 
 
 The model is discretized using flat, triangular shell elements having six degrees of freedom at each node. The discretized model, obtained via \cite{YetAnotherFEcode2020}, has 400 elements and $n = 1320$ DOFs. In both cases, we choose Rayleigh damping, where we tune the mass and stiffness coefficients to achieve a damping ratio of $0.20$ \% for the first two shell modes. To monitor our response predictions, we consider two probe points A and B, as shown in Fig. \ref{fig:vonkarmanshellsconfig}(a). Here, point A is near an antinode position for mode 1 and near a node position for mode 2. The converse holds for point B.  At these probe points, we record the transverse vibration amplitudes $q_A$ and $q_B$. 

We first consider configuration (i), which has the curvature parameter $w = 0.1$ [m]. Similarly to our first example (see section \ref{sec:examplebeam}), we reduce this nonresonant system using the slowest two-dimensional SSM, which is the smoothest nonlinear continuation of the vibration mode shown in \ref{fig:vonkarmanshellsconfig}b. To generate training data, we again follow the first initialization strategy, where we statically force the structure at probe A along the $z$-axis to achieve a deflection of $0.012$ [m].  Indeed, this deflection is more than $40$ \% stiffer compared to the linearized static response. 
\begin{figure}[h!]
    \centering
    \includegraphics[width=1\textwidth]{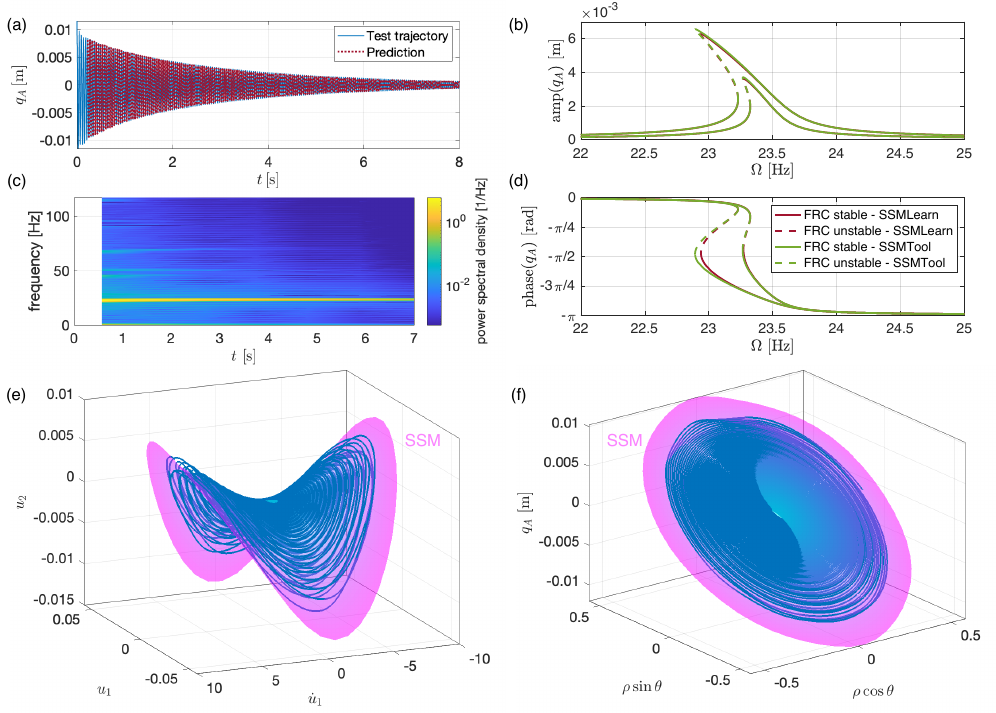}
    \caption{Plot (a) shows the test trajectory from the numerical simulation of the full model (blue curve) in terms of output displacement $q_A$ of the Von Kármán shell in configuration (i) along with its spectrogram in plot (c). Plot (a) also shows the prediction (dark red curve) of the \texttt{SSMLearn} data-driven model. Plots (b,d) respectively show the amplitude and phase for the output displacement $q_A$ of the forced responses for two forcing amplitude values, both for the equation-driven model of \texttt{SSMTool} and for the data-driven model of \texttt{SSMLearn}. Plots (e,f) show instead the parametrization of the two-dimensional SSM along with the training trajectories using two coordinate systems: physical coordinates in (e), while normal form and physical in (f).}
    \label{fig:vonkarmanshell2d}
\end{figure}
For obtaining a test trajectory, we use a slightly lower static load to achieve an initial condition with a deflection of $0.01$ [m], as shown in Fig. \ref{fig:vonkarmanshell2d}a. The spectrogram for this trajectory is shown in Fig. \ref{fig:vonkarmanshell2d}(c), where we observe a strong component near the first natural frequency.  Based on the training trajectory, we compute our SSM-based ROM using polynomials up to order $7$. The reduced dynamics in normal form reads
\begin{equation}\label{eq:nf_vkshelli}
	\begin{array}{rl}\dot{\rho} = &-0.29491\rho-1.2539\rho^{3}+2.6176 \rho^{5}+1.0114 \rho^{7},\\ \dot{\theta} = &+147.4549-35.1372 \rho^{2}+24.7303\rho^{4}+8.5609 \rho^{6}.
	\end{array}
\end{equation}
We validate this unforced ROM against the test trajectory (see  Fig. \ref{fig:vonkarmanshell2d}(a)), where we obtain an optimal NMTE of $6.13$ \%. Fig. \ref{fig:vonkarmanshell2d}(e,f) captures the nonlinear geometry of this autonomous SSM  in two coordinate systems. Plot (e) shows the SSM and the training trajectory attracted to the SSM in the modal coordinates $u_1$, $\dot{u}_1$ and $u_2$, while plot (f) shows them in the polar normal form coordinates $\rho$ and $\theta$ vs. the physical amplitude $q_A$.

Next, we add time-periodic forcing to the above ROM by sinusoidally forcing point A with amplitudes of $10$ and $20$ [N]. We predict the FRCs at these amplitude levels using our SSM-based ROM obtained via \texttt{SSMLearn}, where we expect a nonlinear forced response of softening type~\cite{Jain2021}. In Fig. \ref{fig:vonkarmanshell2d}(b,d), we observe that our data-assisted FRC predictions agree with the equation-based FRC predictions of \texttt{SSMTool}~\cite{Jain2021}. 

%
\begin{figure}[h!]
    \centering
    \includegraphics[width=1\textwidth]{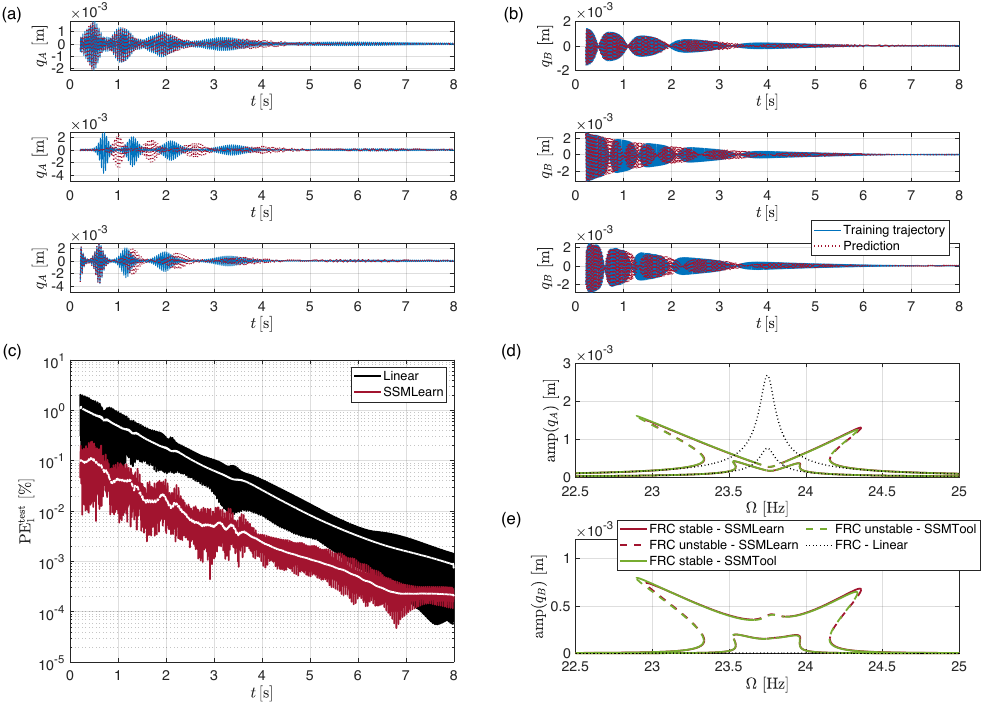}
    \caption{Plot (a) and (b) shows the three training trajectories in terms of the output displacements $q_A$ and $q_B$ for the Von Kármán shell in configuration (ii), i.e., in 1:2 internal resonance. The blue curves depict numerical simulations, while dark red ones model predictions. The trajectories at the top plot in both plots are initialized at the first mode, those in the middle one on the second mode, while the bottom plot are initialized with a linear combination of modes. Plot (c) shows the 1-step prediction error in modal coordinates normalized by the maximum amplitude value for the linear model (black curve) and for the \texttt{SSMLearn} data-driven model (dark red curve). Plots (d,e) show the amplitudes for $q_A$ and $q_B$ of forced responses computed via \texttt{SSMTool} (green) and \texttt{SSMLearn} (dark red) for two forcing amplitude values, also including the linear one in black.}
    \label{fig:vonkarmanshell4d}
\end{figure}

We now consider configuration (ii) with the curvature parameter $w = 0.041$ [m], which results in an approximate $1:2$ internal resonance. In Fig. \ref{fig:vonkarmanshellsconfig}(d,e), we present the mode shapes and the frequencies of the first two modes. To reduce this internally resonant structure, we compute the four-dimensional SSM associated to the resonant modes via \texttt{SSMLearn}. Similarly to our second example (see section \ref{sec:exampleprismatic}), we use the modal initialization strategy of Section \ref{sec:initcond} to generate training and testing trajectories. Specifically, we generate three trajectories for training and three for testing. Each of the two sets of trajectories are obtained by simulating one initial condition along the first mode, one along the second mode, and one along a random convex combination of the two mode shapes, as discussed in Section \ref{sec:exampleprismatic}. We plot the three training trajectories at the probe points $q_A$, $q_B$ in Figs. \ref{fig:vonkarmanshell4d}(a,b). 

Using the training data, we compute a cubic-order, SSM-based ROM. The polar vector field describing the normal form of the reduced dynamics is too long to be reported here, but can be found together with this example in the openly available repository of \texttt{SSMLearn}. 
We compare the predictions of this unforced ROM on the training set, as shown in  Fig. \ref{fig:vonkarmanshell4d} (a,b), where we observe pointwise discrepancies between the ROM predictions and the data. Despite these discrepancies, we obtain a relatively low NMTE of $7.19$ \% on the test data set. As an additional validation step, we compute the 1-step prediction error $\mathrm{PE}_1$ in the test trajectories for our ROM and compare it to that of the linearized ROM. Indeed, this comparison in Fig. \ref{fig:vonkarmanshell4d}(c) shows that $\mathrm{PE}_1$ for our ROM is one order of magnitude lower than that of the linearized ROM. 

We remark that the prediction discrepancies observed in Fig. \ref{fig:vonkarmanshell4d} (a,b) may not necessarily indicate poor training, since simulations of nonlinear systems are sensitive to initial conditions and may diverge over long time scales. Minimizing the error on individual training trajectories in such cases can result in overfitting. A more reliable indicator of good training for the ROM is an accurate prediction of attracting/hyperbolic invariant sets in the nonlinear system, such as stable/unstable periodic orbits, which we consider next. 

We now make FRC predictions using our SSM-based ROM for the same forcing as in configuration (i) except with forcing amplitudes of $2$ [N] and $7$ [N]. In Fig.~\ref{fig:vonkarmanshell4d}(d,e), we observe that the data-assisted predictions of \texttt{SSMLearn} model match closely those obtained using the cubic-order model of \texttt{SSMTool}~\cite{Li2022a}.  In these plots, we also contrast our predictions against the linearized response, which does not exhibit any phenomenon arising from the resonance coupling. 

While we compare our results to those obtained from the equation-driven ROMs via  \texttt{SSMTool}, we note that the validation of FRC computations in both these configurations against full system simulations (e.g., via shooting, collocation, and harmonic balance methods) have already been performed in \cite{Jain2021,Li2022a}.

\subsection{Von Kármán plate in 1:1 internal resonance}

As an additional example of a two-dimensional structure, we discuss the nonlinear vibrations of a simply-supported, square, von Kármán plate model,  originally proposed and validated in \cite{Li2022a}. In this configuration, classical linear plate vibrations give an exact $1:1$ internal resonance between the second and third plate modes (see Fig. \ref{fig:vonkarmanplate}(a-d), where the first four modes are illustrated). Interestingly, the SSM related to these two modes is not the slowest SSM, arising from the continuation of the first bending mode. With this example, we aim to demonstrate the effectiveness of our approach in identifying intermediate SSMs \cite{Haller2016} as well.
\begin{figure}[h!]
    \centering
    \includegraphics[width=1\textwidth]{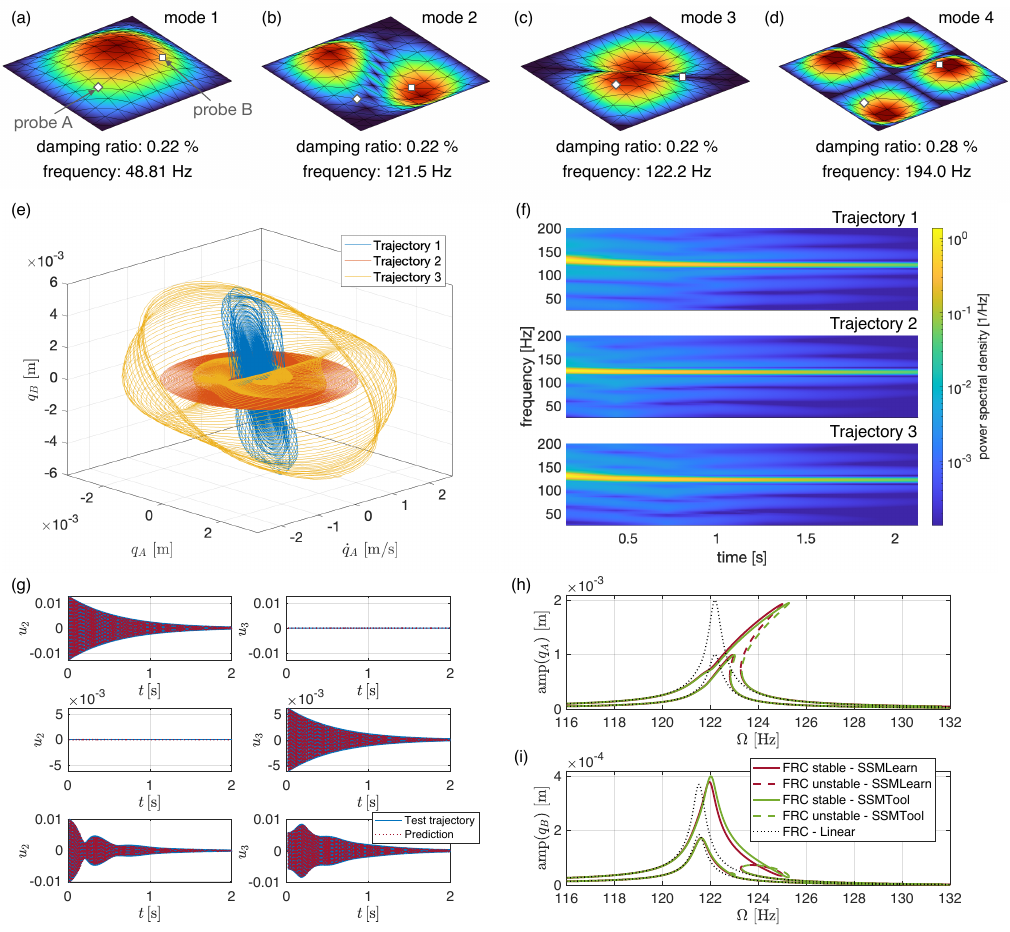}
    \caption{Plots (a-d) show the mode shape as well as the damping ratios and frequencies for the first four plate modes. Plot (a) also indicates the grid position for two output probe nodes A (diamond) and B (square). Plot (e) shows the three decaying trajectories used for model training in physical coordinates, using the transverse displacements $q_A$, $q_B$ and the velocity of the former, whereas plot (f) shows the trajectories spectrogram of $q_A$. Plot (g) illustrates the test trajectories in the two modal coordinates obtained from the full-order model (in blue) and their prediction (dark red) from the data-driven reduced-order model. Plot (h,i) respectively show the amplitudes for $q_A$ and $q_B$ of forced responses computed via \texttt{SSMTool} (green) and \texttt{SSMLearn} (dark red) for two forcing amplitude values, also including the linear forced responses in black.}
    \label{fig:vonkarmanplate}
\end{figure}

By uniformly dividing the plate sides into ten subintervals, $200$ triangular elements are used to discretize the plate, resulting in $606$ DOFs. The FE model is built using the open-source FE package \cite{YetAnotherFEcode2020}. Following~\cite{Li2022a}, we choose the material density as $2700$ [kg/m$^3$], Young's modulus as $70$ [GPa], Poisson's ratio as $0.33$, plate length and width as $1$ [m], and its thickness as $0.01$ [m]. We also choose a Rayleigh damping model with $\mathbf{C} = \mathbf{M} + 4\mathrm{e}{-6}\,\mathbf{K}$, see \cite{Li2022a} for more details. With these system parameters, the mesh and the first four mode shapes, frequencies and damping ratios are illustrated in Fig. \ref{fig:vonkarmanplate}(a-d).  To monitor our response predictions, we consider two probe points A and B, as shown in Fig. \ref{fig:vonkarmanplate}(a-d). We note here that point A is near a node position for mode 2 and near an antinode position for mode 3. The converse holds for point B.  At these probe points, we record the transverse vibration amplitudes $q_A$ and $q_B$.


We aim to model the intermediate, four-dimensional SSM related to the second and third bending modes of the plate, which exhibit an approximate $1:1$ resonance in our FE model. To obtain training trajectories, we again adopt the second initialization strategy described in Section~\ref{sec:initcond}. Similarly to the internally resonant examples of Sections~\ref{sec:exampleprismatic} and \ref{sec:exampleshells}, we use three training trajectories, shown in Fig. \ref{fig:vonkarmanplate}(e), initialized at appropriate amplitudes (i) along mode 2, (ii) along mode 3, and (iii) along a random linear combination of modes 2 and 3.  Inspecting the spectrograms of these trajectories in Fig.~\ref{fig:vonkarmanplate}(f), we can verify that the dominant frequency presence is that of the second and third mode. 

Once again, the test trajectories are initialized similarly but at 2\%  lower amplitudes than the training trajectories. Using a cubic-order parametrization for the SSM and a quintic-order parametrization for its reduced dynamics, we obtain a ROM via \texttt{SSMLearn} with an optimal NMTE of approximately $5$ \%  on the test data. The polar vector field describing the normal form of the reduced dynamics in terms of the polar amplitudes  $\rho_1$, $\rho_2$ and the phase difference $\theta_1-\theta_2$, is too long to be reported here, but can be found together with this example in the openly available repository of \texttt{SSMLearn}. In Fig.~\ref{fig:vonkarmanplate}(g), we obtain good agreement of our test trajectory predictions with the test data along the second and third modal coordinates $u_2$, $u_3$. 

For studying the forced response, we apply a time-periodic (sinusoidal), concentrate load in the transverse direction at point A, shown in Fig.~\ref{fig:vonkarmanplate}(a-d). As this point is near an antinode for mode 3 and a node for mode 2, we expect the modal component of the applied force along mode 3 to be larger than that along mode 2. Due to the $1:1$ resonance, however, we still expect non-trivial interactions between the two modes in the forced response.

We make FRC predictions using our SSM-based ROM at forcing amplitudes of 20 [N] and 40 [N].  In Fig. \ref{fig:vonkarmanplate}(h,i), we observe that the predictions of our data-assisted ROM, which is not trained on any forced data, are consistent with the equation-based predictions obtained using the quintic-order model of \texttt{SSMTool}~\cite{Li2022a}.  In these plots, we also contrast our predictions against the linearized response, which does not exhibit any phenomenon arising from the resonance coupling. Once again, we remark that the validation of these FRC computations against full system simulations (e.g., via shooting, collocation and harmonic balance methods) have already been performed  in \cite{Li2022a}.


\subsection{High-dimensional FE model of a MEMS device}
\label{sec:mems}
\begin{figure}[h!]
    \centering
    \includegraphics[width=1\textwidth]{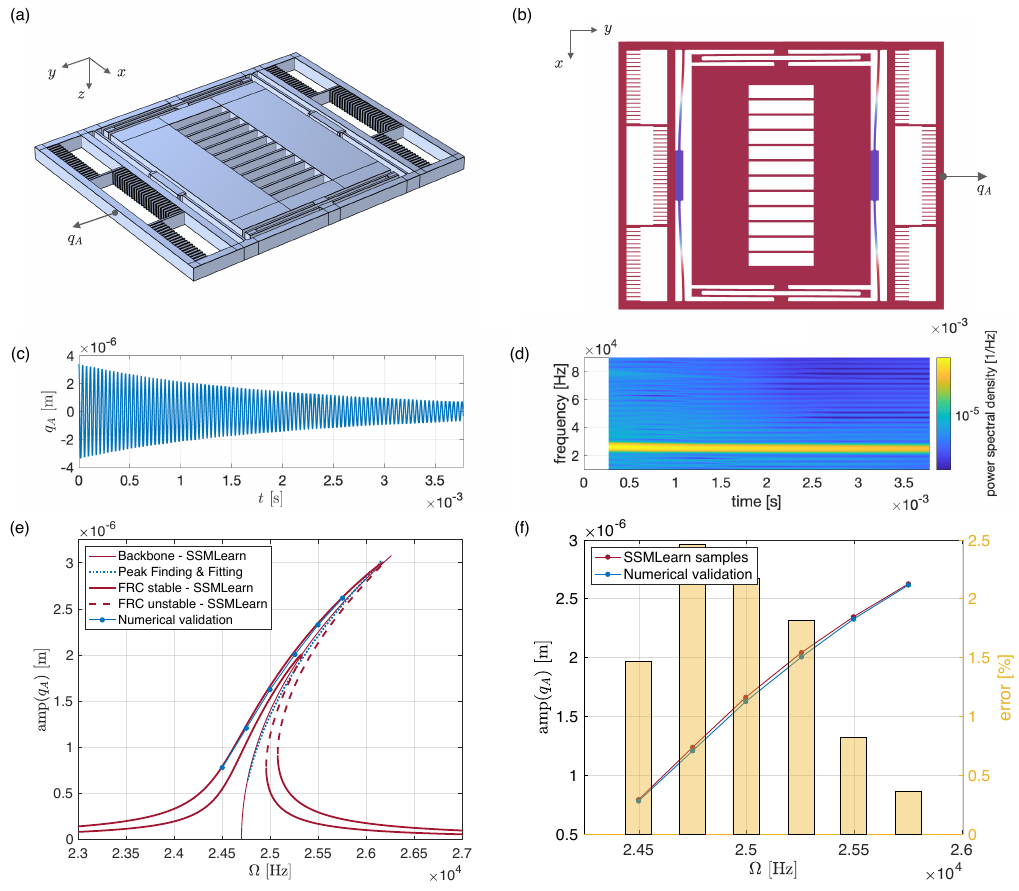}
    \caption{Plots (a-d) show the MEMS device, with the design (a) and the normalized displacement field for the first mode shape (b). In both plots we indicate the reference output amplitude $q_{A}$ as the displacement along the $y$-axis of the outer frame. Plot (c) depicts the training data obtained from the commerical FE code along with its spectrogram in (d) for the output amplitude $q_{A}$. Plot (e) shows the amplitudes for $q_A$ of forced responses computed via \texttt{SSMLearn} for two forcing amplitude values, also including the backbone curve of the \texttt{SSMLearn} model and that obtained by signal processing using the method Peak Finding and Fitting \cite{Jin2020}. Plot (f) depicts the error between some samples of the forced response for the highest forcing amplitude predicted by \texttt{SSMLearn} and their numerical validation.}
    \label{fig:mems}
\end{figure}

As our final example, we study a large FE model representing a MEMS gyroscope prototype~\cite{Marconi2021}, composed of a frame and a proof mass, as shown in Fig.~\ref{fig:mems}a. The overall physical dimensions of the device are 600x600x20 [$\mu$m] The frame is connected to the ground via four flexible beams, allowing its displacement in the y-direction. The proof mass is connected to the frame using two flexible folded beams, which allow the proof mass to move in the $x$-direction. The first vibration mode of the structure (see Fig.~\ref{fig:mems}b) depicts a synchronous motion of the frame and the proof mass in the $y$-direction. In the presence of an external angular rotation, the oscillation along the first mode results in a relative velocity that will generate a Coriolis force on the proof mass along the $x$-direction. This Coriolis force excites the second vibration mode, which comprises the sole motion of the proof mass along the x direction. The parallel plate capacitors within the proof mass detect this motion along the x direction and convert it into an angular velocity measurement. This prototype was designed to exhibit a strongly nonlinear forced response along the first mode so that the drive frequency near the first mode can be tuned to match the sense frequency along the second mode (see~\cite{Marconi2021} for details). This \emph{mode-matched operation} strategy enhances sensing by exploiting the amplification provided by both modes. Hence, in the present analysis, we are interested in the nonlinear response of the first mode (drive mode). To this end, we aim to construct a two-dimensional ROM based on the SSM along the drive mode.

For FE simulations of the full system in this example,  we use the commercial software, COMSOL Multiphysics\textsuperscript{\textregistered} 6.0. Our FE mesh is composed of 28,084 hexahedral elements and 8,636 prismatic (wedge) elements, resulting in 1,029,456 degrees of freedom. Due to the relatively low damping found in MEMS applications, convergence of full system simulations to a steady-state is computationally challenging. In the present example, we consider moderately low damping with a quality factor $Q=200$ of the first mode. To achieve this quality factor, we have employed proportional damping $\mathbf{C} = \alpha \mathbf{M} + \beta \mathbf{K}$, with $\alpha = -829.88$ and $\beta=6.6679\times 10^{-8}$. 


To generate a decaying simulation trajectory, we initialize the system at rest with a displacement along the first vibration mode such that the maximum deflection attained is 3.4 $\mu$m. We simulate the system for a timespan $T=100T_0$, where $T_0$ is the time period of the first undamped natural frequency. We perform time integration in COMSOL Multiphysics\textsuperscript{\textregistered} 6.0 using the generalized alpha solver with the strict time-stepping method, where we choose the time step $dt=T_0/100$, and a relative tolerance of $10^{-5}$. With these settings, the simulation required 10,000 time steps and took about 110 hours. Fig.~\ref{fig:mems} c shows the displacement $q_A$ of the training trajectory along the $y$ -axis at a specific location on the outer frame (see Figs.~\ref{fig:mems}a,b).

In addition to large simulation times, memory requirements pose constraints to storing the trajectory data due to the high-dimensionality of the model. Indeed, each time snapshot of a full solution vector requires approximately 8.2 MB of storage. This would result in an exorbitant memory requirement of approximately 82,000 MB to save all snapshots for the 10,000 time steps of the simulation trajectory. However, this is not an issue for our approach, because full solution snapshots are required only for learning the SSM parametrization, where a high sampling frequency is not useful. Instead of the sampling rate of 100 snapshots per period used for time integration, we choose to store full solution snapshots for only 5 samples per period. This strategy reduces the memory requirements by a factor of 20 and results in similar accuracy in estimating the SSM parametrization. On the other hand, high-resolution snapshots are necessary to have a good estimation of the time derivative of the reduced states, which is used for learning the reduced dynamics. Storing high-resolution snapshots for the reduced states, however, does not pose any memory constraints because of their low dimensionality. Therefore, we use different sampling frequencies for storage depending on variables: the reduced variables are densely sampled in time whereas the full solution vectors are sampled sparsely.




We train our model on the trajectory shown in Fig.~\ref{fig:mems}c. Figure~\ref{fig:mems}d shows the spectrogram for this trajectory along the coordinate $q_A$, where we observe an initial presence of a third harmonic besides the strong signature of the principal frequency. We use a cubic-order model for the geometry and a septic-order model for the reduced dynamics, whose normal form approximation is obtain via SSMLearn as
\begin{equation}\label{eq:nf_mems}
	\begin{array}{rl}\dot{\rho} = &  -775.93\rho -570.52\rho^{3}-109.86\rho^{5}-22.058\rho^{7},\\ \dot{\theta} = & 155183.20+31794.03 \rho^{2}-5697.20 \rho^{4}-1094.84 \rho^{6},\end{array}
\end{equation}
and the resulting NMTE on the observable $q_A$ along the training trajectory amounts to nearly 7\%. We observe that there is a good agreement between the backbone predicted by the model~\eqref{eq:nf_mems} and that extracted by Peak Finding and Fitting in Fig.~\ref{fig:mems}(e). 

We then add forcing, where we use a nodal force at the location of $q_A$ to harmonically excite the frame along the y axis. Finally, we use our SSM-based ROM to make forced response predictions at two forcing amplitudes (approximately 1.94 $\mu$N and 3.32 $\mu$N), where the highest displacement amplitude is predicted around 3 $\mu$m, as shown in Fig.~\ref{fig:mems}(e). The total time spent on the construction of this SSM-based ROM and making forced response predictions was about 5 minutes and 40 seconds.

\begin{table}
    \centering
    \begin{tabular}{|c|c|} \hline 
         Forcing frequency [Hz]&  Simulation time [hours]\\ \hline 
         24,498&  43.67
\\ \hline 
         24,750&  45.01
\\ \hline 
         24,996&  49.19
\\ \hline 
         25,255&  50.10
\\ \hline 
         25,496&  53.82
\\ \hline 
         25,751&  53.09
\\ \hline
    \end{tabular}
    \caption{Computation times for numerical validation of 6 points on the FRC, shown in Fig.~\ref{fig:mems}f, using COMSOL Multiphysics\textsuperscript{\textregistered} 6.0. For comparison purposes, the total time spent on the construction of the SSM-based ROM for the MEMS and making forced response predictions was about 5 minutes and 40 seconds.}
    \label{tab:validation_times}
\end{table}

To validate our predictions, we sample the upper branch of the highest amplitude FRC in Fig. \ref{fig:mems}(e)  at six different frequencies near resonance, as shown in Table~\ref{tab:validation_times}. For these six attractors, we use the initial condition on the periodic response predicted by our SSM-based ROM as input for the time integration of the forced system via  COMSOL Multiphysics\textsuperscript{\textregistered}. We simulate each initial condition for 50 cycles of forcing and expect that the simulated trajectories will remain close to our predicted periodic responses. We plot the NMTE in Fig. \ref{fig:mems}(f), where we observe that the error in our SSM-based prediction relative to the full system simulations is lower than 2.5 \%. The time spent on individual validations is recorded in Table ~\ref{tab:validation_times}.

\section{Conclusion}
In this work, we have developed a data-assisted approach for nonintrusive model reduction of nonlinear mechanical systems based on SSM theory. Our approach uses unforced simulation data of an initially displaced structure to fit an SSM of appropriate dimensionality. Specifically, for a nonresonant structure, a two-dimensional SSM around the fundamental natural frequency governs the nonlinear response. For internally resonant systems, however, higher-dimensional SSMs are necessary, based on the modes participating in the internal resonance, as we have demonstrated. Thus, we have developed a systematic procedure for identifying SSMs and their reduced dynamics in nonresonant as well as internally resonant systems. 

We have shown that the SSM and its reduced dynamics, which are learned from unforced data, can make highly accurate forced response predictions for the full nonlinear system. This is a direct result of SSM theory, which postulates the persistence of SSM under the addition of external forcing to the full system. We have demonstrated accurate predictions of nonlinear forced response for FE models of various mechanical structures comprising beams, shells, and three-dimensional continuum-based elements. 

For very high-dimensional systems, the main computational bottleneck for obtaining these ROMs is the offline cost associated to the FE simulation of decaying trajectories. However, these offline costs of unforced simulations are marginal compared to the comprehensive forced response predictions that are made using the SSM-based ROM. Indeed, for our MEMS resonator example containing more than one million DOFs, the offline cost of obtaining the training trajectory for this ROM was nearly 4.5 days. Using this training data, the SSM-based ROM and the FRC were computed in less than 6 minutes. However, validating the predicted FRC at only six points via full system simulations took more than 12 days.

We fully expect that the data-assisted, nonintrusive SSM reduction developed here will perform equally well under parametric resonance~\cite{Thurnher2023}, which will be pursued in future work.
\section*{Conflicts of interest}
The authors declare conflicts of interest.

\section*{Funding}
No funding was received for conducting this study.

\section*{Data availablity}
All examples are programmed using the open-source package, SSMLearn, available at \url{https://github.com/haller-group/SSMLearn}. The full system simulation data for the examples is available upon request from the authors. 

\begin{appendix}

\section{Non-modal graph style parametrizations}\label{app:nmgs}
In principle, we can choose a coordinate chart as an arbitrary projection to the reduced coordinate defined by a matrix $\mathbf{W}_{0}$, cf. Fig. \ref{fig:Fig2}. By the invertibility of coordinate chart and parametrization, we have that $\mathbf{W}_0\mathbf{V}_0 = \mathbf{I}$, and, by the linear invariance for the parametrization, $\mathbf{V}_0$ must span the spectral subspace $E^{2m}$, i.e., $\mathbf{V}_{0} = \mathbf{V}_{E^{2m}}\mathbf{P}^{-1}$. If we multiply both sides of this identity by $\mathbf{W}_{0}$, we have get the matrix $\mathbf{P}=\mathbf{W}_{0}\mathbf{V}_{E^{2m}}$. This practically means that we are free to choose the reduced-coordinate just using the linear projection via $\mathbf{W}_{0}$, as long as the matrix $\mathbf{P}$ is not singular. This guarantees our parametrization style to be valid at least for a small neighborhood of the origin. The linear part of the dynamics is $\mathbf{R}_{0} = \mathbf{P} \mathbf{R}_{E^{2m}}\mathbf{P}^{-1}$. This concept is useful for example for cases in which one desires to use physically-meaningful coordinates in the parametrization. An example is the POD modes in the vortex shedding example of \cite{Cenedese2022a}, the energy-type variables in \cite{Kaszas2022}. In our context of mechanical system, a possible choice is to use $m$ generic displacement-velocity pairs of the degrees of freedom (or their linear combinations), as long as they are far from being nodes of all the mode shapes related to the SSM. Here, the autonomous reduced dynamics is a reduced mechanical system for the chosen degrees of freedom $\mathbf{q}_m$, i.e., $\mathbf{y} = (\mathbf{q}_m,\dot{\mathbf{q}}_m)$, and such a form can be really insightful from a physical viewpoint. As an example for a two-dim. SSM, if the first degree of freedom is not a node of the mode shape related to the SSM, then $\mathbf{W}_{0}\in\mathbb{R}^{2\times 2n}$ can be chosen to be the projection to the first degree of freedom and its velocity, i.e., 
\begin{equation}\label{eq:nonmodalex}
\begin{cases}
	(\mathbf{W}_{0})_{j,k} = 1 & \mathrm{for\,\,} (j,k)= (1,1), \, (2,n+1) \\
    (\mathbf{W}_{0})_{j,k} = 0 & \mathrm{otherwise.}
\end{cases}
\end{equation}

However, inserting the forcing in the model is not as straightforward as in the case of a modal projection. Indeed, the coordinate chart we choose would not, in general, satisfy the linear invariance $\mathbf{W}_{0}\mathbf{A}=\mathbf{R}_{0}\mathbf{W}_{0}$. So, one needs attention when including forcing, as, for example, Eqs. (\ref{eq:resforcing},\ref{eq:paramforcingnonmodal}) do not hold. By writing $\dot{\mathbf{v}}_{1}(\boldsymbol{\Omega}t) = D\mathbf{v}_{1}(\boldsymbol{\Omega}t)\boldsymbol{\Omega} $, in this case we have that
\begin{equation}
	\begin{array}{l}
	\displaystyle \mathbf{r}_{1}(\boldsymbol{\Omega}t)=\mathbf{W}_{0}\mathbf{A}\mathbf{v}_{1}(\boldsymbol{\Omega}t)+ \mathbf{W}_{0}\mathbf{f}_{1}(\mathbf{0},\boldsymbol{\Omega}t;0), \\
	\displaystyle \dot{\mathbf{v}}_{1}(\boldsymbol{\Omega}t) = (\mathbf{I} - \mathbf{V}_{0}\mathbf{W}_{0})\mathbf{A}\mathbf{v}_{1}(\boldsymbol{\Omega}t) + (\mathbf{I} - \mathbf{V}_{0}\mathbf{W}_{0})\mathbf{f}_{1}(\mathbf{0},\boldsymbol{\Omega}t;0).
	\end{array}
\end{equation}
Therefore, forcing in the reduced dynamics is not simply its modal component, but its form feature additional terms to be derived by solving a linear ODE.
\begin{figure}[t]
    \centering
    \includegraphics[width=.8\textwidth]{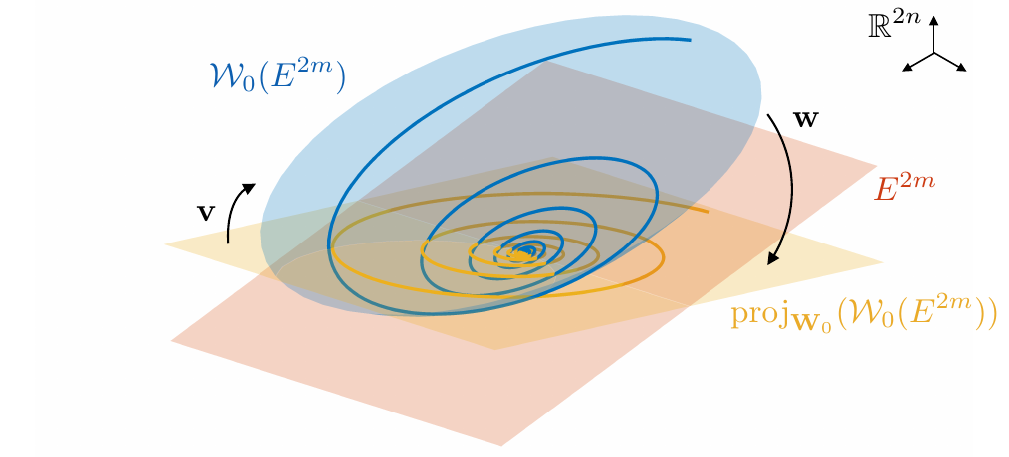}
    \caption{Illustration of parametrization of an autonomous invariant manifold $\mathcal{W}_0(E^{2m})$ (in blue, with a trajectory on it) using the projection to an arbitrary linear subspace space represented by a matrix $\mathbf{W}_{0}$ (in yellow, with the projected trajectory on it). This subspace may be different from $E^{2m}$ but it must be such that the matrix $\mathbf{W}_{0}\mathbf{V}_{E^{2m}}$ is not singular.}
    \label{fig:Fig2}
\end{figure}

\section{Properties of the regressed parametrization}\label{app:reg_para_prop}
Let $\mathbf{v}_{\mathrm{nl}}(\mathbf{y})=\mathbf{V}_{\mathrm{nl}}\mathbf{y}^{2:M}$ where $\mathbf{y}^{2:M}$ is the vector of all $n_{2:M}$ monomials from order 2 to $M$ in $2m$ variables, these being the components of $\mathbf{y}\in\mathbb{C}^{2m}$. The minimization problem of Eq. (\ref{eq:regressparared}) can be then rewritten as
\begin{equation}\label{eq:regressparamat}
	\mathbf{V}_{\mathrm{nl}\star} = \mathrm{arg}\min_{\mathbf{v}_{\mathrm{nl}}}\sum_{j=1}^{P}\left\Vert\mathbf{x}_{j} - \mathbf{V}_0\mathbf{y}_{j}-\mathbf{V}_{\mathrm{nl}}\mathbf{y}^{2:M}_j\right\Vert ^{2}.
\end{equation}
Let us define the matrix $\mathbf{Y}_{\mathrm{nl}}=[\mathbf{y}^{2:M}_1 \,\, \mathbf{y}^{2:M}_2 \,\, ... \,\, \mathbf{y}^{2:M}_P]\in\mathbb{C}^{n_{2:M}\times P}$.
\begin{proposition}
	If the rank of $\mathbf{Y}_{\mathrm{nl}}$ is equal to $n_{2:M}$, the optimal solution in Eq. (\ref{eq:regressparamat}) is unique and always such that $\mathbf{W}_0\mathbf{v}_{\mathrm{nl}}(\mathbf{y})\equiv \mathbf{0}$.
\end{proposition}
\begin{proof}
	One can show that the $\mathbf{V}_{\mathrm{nl}\star}$ can be computed in closed form as the problem is a standard least squares minimization, thereby taking the form
\begin{equation}\label{eq:regressparamatsol}
	\mathbf{V}_{\mathrm{nl}\star} = (\mathbf{X}-\mathbf{V}_0\mathbf{Y}) \mathbf{Y}_{\mathrm{nl}}^H(\mathbf{Y}_{\mathrm{nl}}\mathbf{Y}_{\mathrm{nl}}^H)^{-1},
\end{equation}
where $\mathbf{X}=[\mathbf{x}_1 \,\, \mathbf{x}_2 \,\, ... \,\, \mathbf{x}_P]\in\mathbb{C}^{2n\times P}$ and $\mathbf{Y}=[\mathbf{y}_1 \,\, \mathbf{y}_2 \,\, ... \,\, \mathbf{y}_P]=\mathbf{W}_0\mathbf{X}\in\mathbb{C}^{2m\times P}$. If the rank of $\mathbf{Y}_{\mathrm{nl}}$ is equal to $n_{2:M}$, then the square matrix $\mathbf{Y}_{\mathrm{nl}}\mathbf{Y}_{\mathrm{nl}}^H$ is invertible and $\mathbf{V}_{\mathrm{nl}\star}$ in Eq. (\ref{eq:regressparamatsol}) is the unique least squares solution. Multiplying Eq. (\ref{eq:regressparamatsol}) by $\mathbf{W}_0$ and recalling that $\mathbf{W}_0\mathbf{V}_0=\mathbf{I}$, we conclude that
\begin{equation}\label{eq:regressparamatsol1}
	\mathbf{W}_0\mathbf{V}_{\mathrm{nl}\star} = (\mathbf{W}_0\mathbf{X}-\mathbf{W}_0\mathbf{V}_0\mathbf{Y}) \mathbf{Y}_{\mathrm{nl}}^H(\mathbf{Y}_{\mathrm{nl}}\mathbf{Y}_{\mathrm{nl}}^H)^{-1}= (\mathbf{W}_0\mathbf{X}-\mathbf{Y}) \mathbf{Y}_{\mathrm{nl}}^H(\mathbf{Y}_{\mathrm{nl}}\mathbf{Y}_{\mathrm{nl}}^H)^{-1}=\mathbf{0}.
\end{equation}
\end{proof}
We note that the same orthogonality relation also holds for weighted ridge regression, optionally available in \texttt{SSMLearn} \cite{Cenedese2022a}.

\section{External periodic forcing in the normal form}\label{app:s1}
Let us denote $\boldsymbol{\zeta} = \mathbf{P}^{-1}\mathbf{y} \in\mathbb{C}^{2m}$. Using these new coordinates, we have that
\begin{equation}
	\label{eq:coordcnorformzeta}
	\begin{array}{l}
		\boldsymbol{\zeta}= \mathbf{P}^{-1}\mathbf{h}(\mathbf{z},\boldsymbol{\Omega}t;\varepsilon)=\hat{\mathbf{h}}(\mathbf{z},\boldsymbol{\Omega}t;\varepsilon)=\mathbf{z} + \hat{\mathbf{h}}_{\mathrm{nl}}(\mathbf{z})-\varepsilon\mathbf{h}_{1}(\boldsymbol{\Omega}t), \\ \mathbf{z}= \mathbf{h}^{-1}(\mathbf{P}\boldsymbol{\zeta},\boldsymbol{\Omega}t;\varepsilon)=\hat{\mathbf{h}}^{-1}(\boldsymbol{\zeta},\boldsymbol{\Omega}t;\varepsilon)=\boldsymbol{\zeta} + \hat{\mathbf{h}}^{-1}_{\mathrm{nl}}(\boldsymbol{\zeta})+\varepsilon\mathbf{h}_{1}(\boldsymbol{\Omega}t), \\
		\dot{\boldsymbol{\zeta}}=\mathbf{P}^{-1}\mathbf{r}(\mathbf{P}\boldsymbol{\zeta},\boldsymbol{\Omega}t;\varepsilon)=\hat{\mathbf{r}}(\boldsymbol{\zeta},\boldsymbol{\Omega}t;\varepsilon)=\mathbf{R}_{E^{2m}}\boldsymbol{\zeta} + \hat{\mathbf{r}}_{\mathrm{nl}}(\boldsymbol{\zeta}) + \varepsilon\hat{\mathbf{r}}_{1}(\boldsymbol{\Omega}t),
	\end{array}
\end{equation}
where, for the case of periodic forcing introduced in Eq. (\ref{eq:externalperiodicforcing}), we have that $\hat{\mathbf{r}}_{1}(\Omega t) = \mathbf{g}^r \left( e^{i\Omega t}+e^{-i\Omega t} \right)$. By Fourier analysis, we can write
\begin{equation}
	\begin{array}{c}
	\mathbf{n}_{1}(\Omega t)
	= \mathbf{g}^{n+} e^{i\Omega t}+\mathbf{g}^{n-}e^{-i\Omega t} , \qquad \mathbf{g}^{n\pm} = \left( g^{n\pm}_1, \, g^{n\pm}_2, \,  ... \,, \, g^{n\pm}_m, \, \bar{g}^{n\pm}_1, \, \bar{g}^{n\pm}_2, \, ... \,, \, \bar{g}^{n\pm}_m \right)^\top, \\
	\mathbf{h}_{1}(\Omega t)
	= \mathbf{g}^{h+} e^{i\Omega t}+\mathbf{g}^{h-} e^{-i\Omega t}, \qquad \mathbf{g}^{h\pm} = \left( g^{h\pm}_1, \, g^{h\pm}_2, \,  ... \,, \, g^{h\pm}_m, \, \bar{g}^{h\pm}_1, \, \bar{g}^{h\pm}_2, \, ... \,, \, \bar{g}^{h\pm}_m \right)^\top,
	\end{array}
\end{equation}
and, by considering the $\mathcal{O}(\varepsilon)$-term in the conjugacy equation $D_\mathbf{z} \hat{\mathbf{h}}(\mathbf{z},\boldsymbol{\Omega}t;\varepsilon) \mathbf{n}(\mathbf{z},\boldsymbol{\Omega}t;\varepsilon) + $ \linebreak $D_{\boldsymbol{\Omega}t} \hat{\mathbf{h}}(\mathbf{z},\boldsymbol{\Omega}t;\varepsilon)\boldsymbol{\Omega} =\hat{\mathbf{r}}(\hat{\mathbf{h}}(\mathbf{z},\boldsymbol{\Omega}t;\varepsilon),\boldsymbol{\Omega}t;\varepsilon)$, we obtain
\begin{equation}
	\begin{array}{c}
	g^{n+}_k e^{i\Omega t}+g^{n-}_ke^{-i\Omega t} - i\Omega g^{h+}_k e^{i\Omega t}+i\Omega g^{h-}_k e^{-i\Omega t} =  -\lambda_{j_k} g^h_ke^{i\Omega t}-\lambda_{j_k} g^{h-}_k e^{-i\Omega t}  + g_k\left( e^{i\Omega t}+e^{-i\Omega t} \right) , \\
	\bar{g}^{n+}_k e^{i\Omega t}+\bar{g}^{n-}_ke^{-i\Omega t} - i\Omega \bar{g}^{h+}_k e^{i\Omega t}+i\Omega \bar{g}^{h-}_k e^{-i\Omega t} = -\bar{\lambda}_{j_k} \bar{g}^h_ke^{i\Omega t}-\bar{\lambda}_{j_k} \bar{g}^{h-}_k e^{-i\Omega t}  + \bar{g}_k\left( e^{i\Omega t}+e^{-i\Omega t} \right) ,
	\end{array}
\end{equation}
for $k = 1, \,2,\, ... \, , m$. If we solve for the coefficients of the change of coordinates we get
\begin{equation}
	g^{h\pm}_k = \frac{ g_k- g^{n\pm}_k}{\lambda_{j_k} \mp i\Omega }, \qquad \bar{g}^{h\pm}_k = \frac{ \bar{g}_k-\bar{g}^{n\pm}_k}{\bar{\lambda}_{j_k} \mp i\Omega },
\end{equation}
and we clearly see that when $k\in R$ there will be small denominator, as this correspond to resonant forcing for which $\mathrm{Im}(\lambda_{j_k}) \approx \Omega$. Hence, we choose to keep only such resonant forcing terms in the normal form dynamics, thereby leading to
\begin{equation}
	\begin{cases}
	\displaystyle g^{h+}_k = g^{n-}_k = \bar{g}^{n+}_k = \bar{g}^{h-}_k = 0, \,\, g^{n+}_k = g_k,\,\,\bar{g}^{n-}_k = \bar{g}_k,\,\, g^{h-}_k = \frac{g_k}{\lambda_{j_k} + i\Omega },\,\, \bar{g}^{h+}_k = \bar{g}^{h-}_k, \,\, \mathrm{if\,\,} k\in R \\ 
	\displaystyle g^{h\pm}_k = \frac{ g_k}{\lambda_{j_k} \mp i\Omega }, \,\, \bar{g}^{h\pm}_k = \frac{ \bar{g}_k}{\bar{\lambda}_{j_k} \mp i\Omega },\,\, \mathrm{otherwise}.	
	\end{cases}
\end{equation}
\end{appendix}
\bibliographystyle{naturemag}
\bibliography{biblio}
\end{document}